\ifdefined\AddToHook
  \makeatletter
  \let\qmat@kernellabel\label
  \let\qmat@kernelrefstepcounter\refstepcounter
  \AddToHook{package/hyperref/before}{%
    \let\label\qmat@kernellabel
    \let\refstepcounter\qmat@kernelrefstepcounter
  }
  \makeatother
\fi

\documentclass[final,onefignum,onetabnum]{siamart171218}

\usepackage{comment}

\newcommand{\R}{\mathbb{R}}
\newcommand{\LR}{\mathcal{L}_{\mathrm{lr}}}
\newcommand{\Dense}{\mathcal{L}_{\mathrm{d}}}
\newcommand{\Q}{\mathcal{Q}}

\newcommand{\normmax}[1]{\left\lVert #1\right\rVert_{\max}}
\newcommand{\norminf}[1]{\left\lVert #1\right\rVert_{\infty}}
\newcommand{\nuc}[1]{\left\lVert #1\right\rVert_{*}}
\newcommand{\sideeqtag}[1]{%
  \refstepcounter{equation}%
  \label{#1}%
  \textup{(\theequation)}%
}

\usepackage{enumitem}
\usepackage{microtype}
\usepackage{amsmath, amssymb}
\usepackage{lipsum}

\usepackage{amsfonts}
\usepackage{booktabs}
\usepackage{comment}
\usepackage{mathtools}
\usepackage{subfigure}
\usepackage{comment}
\usepackage{mathrsfs}
\usepackage{graphicx}
\usepackage{epstopdf}
\usepackage{algorithmic}
\usepackage{graphicx}
\usepackage{tikz}       
\usepackage{amsmath}    
\usepackage{graphicx}  
\usepackage{amsmath}
\usepackage{algorithmic}

\usetikzlibrary{positioning} 
\usetikzlibrary{arrows.meta} 
\usetikzlibrary{calc}

\numberwithin{equation}{section}
\numberwithin{table}{section}
\numberwithin{figure}{section}

\ifpdf
  \DeclareGraphicsExtensions{.eps,.pdf,.png,.jpg}
\else
  \DeclareGraphicsExtensions{.eps}
\fi

\newsiamremark{remark}{Remark}
\newsiamremark{hypothesis}{Hypothesis}
\newsiamremark{assumption}{Assumption}
\crefname{hypothesis}{Hypothesis}{Hypotheses}
\newsiamthm{claim}{Claim}

\newcommand{\Bset}{\mathcal{B}}

\usepackage{scalerel}

\let\bigopsize\bigoplus
\def\bigoplus{{\scalerel*{\boldsymbol\oplus}{\bigopsize}}}

\newcommand{\fl}{\text{\texttt{fl}}}

\DeclarePairedDelimiter\abs{\text{abs}(}{)}

\headers{Integer-Quantized Hierarchical Matrices}{Erin Carson, Xinye Chen, and Ritesh Khan}

\title{Integer-Quantized Hierarchical Matrices with Adaptive Bit Allocation\thanks{The first and third authors are supported by the European Union (ERC, inEXASCALE, 101075632).
Views and opinions expressed are those of the authors only and do not necessarily reflect those of the European Union or the European Research Council. Neither the European Union nor the granting authority can be held responsible for them. The first author is additionally supported by the Charles University Research Centre program No. UNCE/24/SCI/005.}}

\author{
Erin Carson\thanks{Department of Numerical Mathematics, Charles University, Prague, Czech Republic
  (\email{carson@karlin.mff.cuni.cz}).}
\and Xinye Chen\thanks{Sorbonne Universit\'e, CNRS, LIP6, Paris, France
  (\email{xinye.chen@lip6.fr}).}
\and Ritesh Khan\footnotemark[1]\thanks{Department of Numerical Mathematics, Charles University, Prague, Czech Republic (\email{ritesh.khan@matfyz.cuni.cz}).}
}

\usepackage{amsopn}

\makeatletter
\newcommand*{\addFileDependency}[1]{
  \typeout{(#1)}
  \@addtofilelist{#1}
  \IfFileExists{#1}{}{\typeout{No file #1.}}
}
\makeatother

\ifpdf
\hypersetup{
  pdftitle={Integer-Quantized Hierarchical Matrices with Adaptive Bit Allocation},
  pdfauthor={E. Carson, X. Chen, and R. Khan}
}
\fi

\begin{document}

\maketitle

\begin{abstract}
We present an adaptive bitwidth allocation method for quantizing low-rank factors in hierarchical matrices in order to decrease storage while meeting a given constraint on their reconstruction error. For the bitwidth allocation we formulate a discrete storage minimization problem and show that an optimal assignment can give larger singular values at least as many bits as smaller ones. We further develop a greedy algorithm that meets the specified max-norm quantization error budget. We investigate a hierarchical block partition in which all of the non-diagonal blocks are represented by quantized low-rank factors and only dense blocks remain dense; in this way a global reconstruction bound related to $\eta_k$ and $R_k$ is obtained, with the former and latter being bounds on the truncation and reconstruction arithmetic errors, respectively.   Extensive simulations, carried out with kernel and BEM matrices under block-relative budgets, indicate that the proposed method achieves a reduction in storage of up to $44.3\times$ when compared to storing the complete matrix densely in FP64, and a reduction of up to $5.70\times$ when compared to the FP64 low-rank factors within the same partition, together with measured reconstruction errors that remained below the maximum norm bounds in all the hierarchical tests.

\end{abstract}



\section{Introduction}
Hierarchical matrices were introduced by Hackbusch as a data-sparse approximation of dense boundary element matrices ~\cite{hackbusch1999sparse}. Even after low-rank compression, the stored factors can account for a substantial part of the representation. Reducing their storage precision therefore raises the question of how to control the resulting reconstruction error.
Quantization is the process of mapping continuous infinite values to a smaller set of discrete finite values, typically with reduced bitwidth number formats. For large-scale model training and low-precision inference on NVIDIA GPUs, reduced precision formats such as FP8, FP16 are increasingly used as they strike a better balance between numerical range, precision and hardware efficiency when in proper use. However, integer formats like INT8 and INT4 remain highly relevant on edge devices, CPUs, NPUs, integer-oriented accelerators and in any setting that demands strict integer-only inference, aggressive weight compression, or operates under tight bandwidth constraints. Integer arithmetic still offers clear practical advantages in that multiply-accumulate operations can stay exact during accumulation provided overflow is avoided, asymmetric quantization can make fuller use of the available code range for non-negative or skewed distributions, and an integer-only workflow can eliminate floating-point units altogether, which is especially valuable for power-efficient and resource-constrained hardware. Thus, quantization methods built around integer arithmetic continue to be a useful and active research direction.

Quantization has been verified to significantly speed up deep learning inference. Deep neural networks deployed on edge devices, such as mobile or IoT devices, are subject to stringent compute and memory constraints, while networks deployed in the cloud do not suffer such constraints but might suffer excessive latency and power consumption. The common approach to alleviating the burden of training and inference is via quantization \cite{li2023ivit,kim2021ibert}, and integer quantizations are the most commonly used. SwitchBack \cite{10.5555/3666122.3666573} achieves a speedup of 13-25\% using INT8 quantized training against the performance of bfloat16 training within 0.1 percentage points for the 1B parameter CLIP ViT-Huge.   Details can be found in the survey \cite{gholami2021survey} and references therein. 

The $\mathcal{H}$-matrices are constructed by recursively partitioning a matrix into blocks (sub-matrices) and approximating suitable off-diagonal blocks by low-rank matrices. Several hierarchical matrix formats have been extensively studied in the literature, e.g., $\mathcal{H}$ \cite{hackbusch1999sparse,hackbusch2000sparse}, $\mathcal{H}^2$ \cite{borm2009construction}, HODLR \cite{ambikasaran2019hodlrlib}, HSS \cite{chandrasekaran2007fast}, and HBS \cite{gillman2012direct} formats. In this work, we consider hierarchical matrix formats in which the low-rank factors are formed independently in a non-nested fashion. Advances in low-rank approximation techniques and hierarchical matrix formats have significantly reduced the storage requirements of dense matrices by exploiting their underlying low-rank structure. However, the low-rank factors or generators  in these representations are typically stored in double precision. This may be inefficient from a storage perspective and limit the potential for further storage savings. There has recently been significant work in reduced-precision storage of and computation with hierarchical matrices: Amestoy et al.~\cite{10.1093/imanum/drac037} use singular-value-dependent floating-point precisions for low-rank factors and analyze their use in BLR LU factorization. Kriemann~\cite{kriemann2025hierarchical} studies adaptive compression of low-rank and dense data in hierarchical matrices and its effect on matrix arithmetic. Other reduced-precision hierarchical representations are studied in \cite{ooi2020effect,carson2025mixedhodlr,khan2026hybrid}.

In this work, we develop a singular value aware bit-allocation strategy for the non-diagonal blocks of a hierarchical matrix.  Both admissible blocks and non-diagonal inadmissible blocks use quantized low-rank factors, with only diagonal blocks remaining dense.  Spatial adaptivity is supplied by the hierarchical block partition and its block ranks, whereas precision adaptivity is applied inside each factorized block.  Let $T=U\Sigma V^T$ be a rank-$r$ SVD with $\Sigma=\operatorname{diag}(\sigma_1,\ldots,\sigma_r)$, and let $b_j$ be the bitwidth assigned to component $j$. The local representation uses the balanced factorization $T=(U\Sigma^{1/2})(V\Sigma^{1/2})^T$, and quantizes each balanced column independently with its own symmetric integer scale.  Under this model the mixed-bit quantization error bound for one block is
\begin{equation*}
  E_{\mathrm{cert}}=\sum_{j=1}^r \sigma_j c(b_j),
  \qquad
  c(b)=\frac{1}{q_b}+\frac{1}{4q_b^2},
  \qquad q_b=2^{b-1}-1.
\end{equation*}

Let $B^{\mathrm q}\ge0$ denote the prescribed absolute max-norm error budget for factor quantization, and the superscript $\mathrm q$ indicates quantization. Our greedy allocation accounts for the minimum error of the components that have not yet been assigned a bitwidth. We use $\widehat T$ to denote the exact product of the reconstructed quantized factors. Theorem~\ref{thm:greedy-feasible} proves that, whenever the highest available precision is feasible, the returned local allocation satisfies
\begin{equation*}
  \normmax{T-\widehat T}\le E_{\mathrm{cert}}\le B^{\mathrm q}.
\end{equation*}

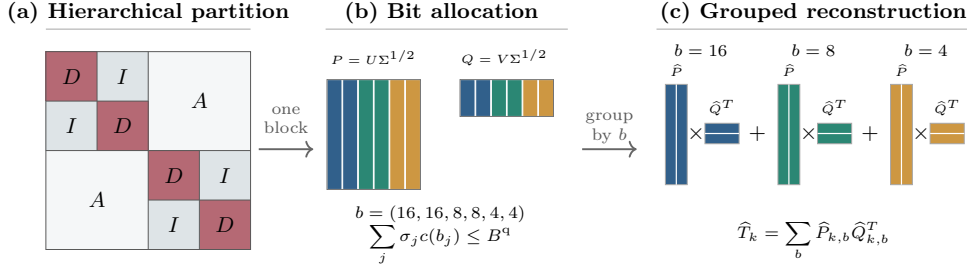
\begin{figure}[tbp]
\centering
\definecolor{blockDiag}{RGB}{181,105,116}
\definecolor{blockNear}{RGB}{222,228,231}
\definecolor{blockFar}{RGB}{245,247,248}
\definecolor{bitSixteen}{RGB}{49,96,139}
\definecolor{bitEight}{RGB}{43,135,116}
\definecolor{bitFour}{RGB}{205,151,66}

\resizebox{\linewidth}{!}{
\begin{tikzpicture}[every node/.style={font=\small}]

\node[font=\small\bfseries] at (1.9,4.12) {(a) Hierarchical partition};
\node[font=\small\bfseries] at (6.2,4.12) {(b) Bit allocation};
\node[font=\small\bfseries] at (11.65,4.12) {(c) Grouped reconstruction};
\draw[black!35] (0.4,3.88)--(3.4,3.88);
\draw[black!35] (4.5,3.88)--(8.05,3.88);
\draw[black!35] (9.35,3.88)--(13.9,3.88);

\fill[blockFar] (0.4,0.65) rectangle (3.4,3.55);
\fill[blockDiag] (0.4,2.825) rectangle (1.15,3.55);
\fill[blockDiag] (1.15,2.1) rectangle (1.9,2.825);
\fill[blockDiag] (1.9,1.375) rectangle (2.65,2.1);
\fill[blockDiag] (2.65,0.65) rectangle (3.4,1.375);
\fill[blockNear] (1.15,2.825) rectangle (1.9,3.55);
\fill[blockNear] (0.4,2.1) rectangle (1.15,2.825);
\fill[blockNear] (2.65,1.375) rectangle (3.4,2.1);
\fill[blockNear] (1.9,0.65) rectangle (2.65,1.375);

\draw[black!65] (0.4,0.65) rectangle (3.4,3.55);
\draw[black!65] (1.9,0.65)--(1.9,3.55);
\draw[black!65] (0.4,2.1)--(3.4,2.1);
\draw[black!65] (1.15,2.1)--(1.15,3.55);
\draw[black!65] (0.4,2.825)--(1.9,2.825);
\draw[black!65] (2.65,0.65)--(2.65,2.1);
\draw[black!65] (1.9,1.375)--(3.4,1.375);

\node at (0.775,3.187) {$D$};
\node at (1.525,2.462) {$D$};
\node at (2.275,1.737) {$D$};
\node at (3.025,1.012) {$D$};
\node at (1.525,3.187) {$I$};
\node at (0.775,2.462) {$I$};
\node at (3.025,1.737) {$I$};
\node at (2.275,1.012) {$I$};
\node at (2.65,2.825) {$A$};
\node at (1.15,1.375) {$A$};

\draw[->,black!60,line width=0.8pt]
  (3.52,2.12)--(4.32,2.12);

\node[
  font=\scriptsize,
  text=black!65,
  align=center
] at (3.92,2.55) {one\\block};

\node[font=\tiny] at (5.2,3.47) {$P=U\Sigma^{1/2}$};
\node[font=\tiny] at (7.12,3.47) {$Q=V\Sigma^{1/2}$};

\filldraw[fill=bitSixteen,draw=white] (4.52,1.53) rectangle (4.74,3.14);
\filldraw[fill=bitSixteen,draw=white] (4.75,1.53) rectangle (4.97,3.14);
\filldraw[fill=bitEight,draw=white]   (4.98,1.53) rectangle (5.20,3.14);
\filldraw[fill=bitEight,draw=white]   (5.21,1.53) rectangle (5.43,3.14);
\filldraw[fill=bitFour,draw=white]    (5.44,1.53) rectangle (5.66,3.14);
\filldraw[fill=bitFour,draw=white]    (5.67,1.53) rectangle (5.89,3.14);

\filldraw[fill=bitSixteen,draw=white] (6.47,2.60) rectangle (6.69,3.14);
\filldraw[fill=bitSixteen,draw=white] (6.70,2.60) rectangle (6.92,3.14);
\filldraw[fill=bitEight,draw=white]   (6.93,2.60) rectangle (7.15,3.14);
\filldraw[fill=bitEight,draw=white]   (7.16,2.60) rectangle (7.38,3.14);
\filldraw[fill=bitFour,draw=white]    (7.39,2.60) rectangle (7.61,3.14);
\filldraw[fill=bitFour,draw=white]    (7.62,2.60) rectangle (7.84,3.14);
\draw[black!55] (6.47,2.60) rectangle (7.84,3.14);

\draw[black!55] (4.52,1.53) rectangle (5.89,3.14);
\node[font=\scriptsize] at (6.18,1.18)
  {$b=(16,16,8,8,4,4)$};
\node[font=\scriptsize] at (6.18,0.72)
  {$\displaystyle\sum_j\sigma_jc(b_j)\le B^{\mathrm q}$};

\draw[->,black!60,line width=0.8pt]
  (8.26,2.12)--(9.06,2.12);

\node[
  font=\scriptsize,
  text=black!65,
  align=center
] at (8.66,2.43) {group\\by $b$};

\node[font=\scriptsize] at (10,3.57) {$b=16$};
\node[font=\tiny] at (9.64,3.28) {$\widehat P$};
\node[font=\tiny] at (10.31,2.73) {$\widehat Q^{T}$};
\filldraw[fill=bitSixteen,draw=black!35]
  (9.48,1.60) rectangle (9.80,3.07);
\draw[white,line width=0.5pt] (9.64,1.60)--(9.64,3.07);
\filldraw[fill=bitSixteen,draw=black!35]
  (10.06,2.20) rectangle (10.56,2.50);
\draw[white,line width=0.5pt] (10.06,2.35)--(10.56,2.35);
\node at (9.93,2.35) {$\times$};

\node[font=\scriptsize] at (11.65,3.57) {$b=8$};
\node[font=\tiny] at (11.29,3.28) {$\widehat P$};
\node[font=\tiny] at (11.96,2.73) {$\widehat Q^{T}$};
\filldraw[fill=bitEight,draw=black!35]
  (11.13,1.60) rectangle (11.45,3.07);
\draw[white,line width=0.5pt] (11.29,1.60)--(11.29,3.07);
\filldraw[fill=bitEight,draw=black!35]
  (11.71,2.20) rectangle (12.21,2.50);
\draw[white,line width=0.5pt] (11.71,2.35)--(12.21,2.35);
\node at (11.58,2.35) {$\times$};

\node[font=\scriptsize] at (13.30,3.57) {$b=4$};
\node[font=\tiny] at (12.94,3.28) {$\widehat P$};
\node[font=\tiny] at (13.61,2.73) {$\widehat Q^{T}$};
\filldraw[fill=bitFour,draw=black!35]
  (12.78,1.60) rectangle (13.10,3.07);
\draw[white,line width=0.5pt] (12.94,1.60)--(12.94,3.07);
\filldraw[fill=bitFour,draw=black!35]
  (13.36,2.20) rectangle (13.86,2.50);
\draw[white,line width=0.5pt] (13.36,2.35)--(13.86,2.35);
\node at (13.23,2.35) {$\times$};

\node at (10.83,2.35) {$+$};
\node at (12.48,2.35) {$+$};
\node[font=\scriptsize] at (11.65,0.78)
  {$\displaystyle\widehat T_k
    =\sum_b\widehat P_{k,b}\widehat Q_{k,b}^{T}$};

\end{tikzpicture}%
}

\caption{Overview of the hierarchical representation and local bit allocation. In (a), $D$ denotes a dense diagonal block, while admissible blocks $A$ and non-diagonal inadmissible blocks $I$ both use quantized low-rank factors. In (b), paired columns receive the same bitwidth and retain separate left and right scales. In (c), columns are grouped by bitwidth for reconstruction. Colors indicate bitwidths; the widths shown are illustrative.} 
\label{fig:method-overview}
\end{figure}

In the hierarchical setting, the quantization budget is the total block budget minus bounds for low-rank approximation and reconstruction arithmetic errors, as defined in~\eqref{eq:quant-budget}. The local result is lifted to the complete $\mathcal H$-matrix by the disjoint partition of matrix entries into blocks.  In entrywise max norm, local block errors combine by a maximum rather than a sum. We further group components assigned the same bitwidth for reconstruction, with each column retaining its own quantization scales. We account for rounding in the corresponding GEMM (General Matrix Multiply) operations and in the summation of their outputs. This explicitly separates three error sources, namely, low-rank construction, factor quantization, and finite-precision GEMM.

Our method is motivated by techniques used for model acceleration  in integer arithmetic and data compression techniques in data-sparse matrices.  Integer quantization is a well-established method for reducing the bitwidth of the number format{--}by converting floating-point numbers into integers{--}such that storage and computational costs can be reduced. Often, integer-based arithmetic, such as integer matrix multiplication, can run much faster than conventional floating-point arithmetic on hardware with dedicated integer SIMD or dot-product support (e.g., x86 CPUs with AVX2/AVX-512/VNNI and ARM CPUs with NEON/DotProd extensions), and particularly in specialized hardware implementations (e.g., UNPU \cite{8310262}, Eyeriss \cite{7551407},  Pragmatic \cite{10.1145/3123939.3123982}) and specialized hardware instructions (e.g., NVIDIA Tensor cores).  Related surveys on neural network quantization can be found in \cite{10.1145/3623402, 10028742}. Our software is publicly available at \url{https://github.com/inEXASCALE/qhmat}.

The remainder of the paper is organized as follows.  Section~\ref{sec:h-setting} gives an overview of hierarchical matrices; Section~\ref{sec:quant_models} formulates the foundations of our quantization models using integer arithmetic and  Section~\ref{sec:allocation} proposes a bit allocation strategy for matrix computations of hierarchical matrices. Section~\ref{sec:assumptions-discussion} provides the storage analysis of our bit allocation strategy. Section~\ref{sec:exps} analyzes the empirical results of experimental simulations, and Section~\ref{sec:conclude} concludes the paper.

\section{Hierarchical matrices}
\label{sec:h-setting}

Hierarchical matrix (in short, $\mathcal{H}$-matrix) formats are block low-rank approximations of dense matrices that exploit the numerical low-rank structure of suitable off-diagonal blocks, which are identified based on an admissibility condition \cite{borm2003hierarchical,weak_hackbusch2004}. $\mathcal{H}$-matrices reduce both the storage requirements and the arithmetic complexity~\cite{hackbusch1999sparse,bebendorf2008hierarchical,hackbusch2015hierarchical}. This paper restricts attention to the singular value decomposition (SVD) for low-rank approximation in hierarchical matrices. Let $M\in\R^{m\times n}$ be a matrix block and $T$ a rank-$r$ approximation represented by its SVD.  Then we have
\begin{equation}
  T = U\Sigma V^T
    = \sum_{j=1}^{r} \sigma_j u_j v_j^T,
  \label{eq:svd}
\end{equation}
where $ \sigma_1\ge \sigma_2\ge\cdots\ge\sigma_r\ge 0$ with $U^TU=I_r$ and $V^TV=I_r$. Here $U\in\R^{m\times r}$, $V\in\R^{n\times r}$, and $\Sigma=\operatorname{diag}(\sigma_1,\ldots,\sigma_r)$.  In the quantization analysis, its residual $M-T$ relative to $M$ is bounded separately. 

Let $I=\{1,\ldots,N\}$ be the matrix index set and $A\in\R^{I\times I}$ have a common row and column cluster tree. The symbol $\dot\cup$ denotes disjoint union. Then, the induced block partition has block index set $\mathcal L=\LR\,\dot\cup\,\Dense,$ where $\Dense$ contains only the diagonal blocks, with identical row and column index clusters, and $\LR$ contains all non-diagonal blocks.  The latter includes both geometrically admissible blocks and non-diagonal inadmissible blocks whose row and column clusters are terminal.  These two classes are distinguished in the storage diagnostics, but use the same quantization model. Admissible blocks may stop the block-tree recursion before both row and column clusters are terminal; geometric admissibility and the choice of storage format are distinct.

Further, each index $k\in\mathcal L$ identifies a block $\tau_k\times\sigma_k$, where $\tau_k,\sigma_k\subseteq I$ are its row and column index clusters.  Set $m_k=|\tau_k|$ and $n_k=|\sigma_k|$. The blocks form the disjoint partition
\begin{equation}
  I\times I
  =\mathop{\dot\bigcup}_{k\in\mathcal L}
   (\tau_k\times\sigma_k).
  \label{eq:leaf-partition}
\end{equation}
Consequently, each matrix entry belongs to exactly one block.  This disjointness is a property of the block partition rather than a synonym for non-diagonality or geometric inadmissibility.  Off-diagonal row and column index clusters are disjoint even when their geometric supports touch.

For every $k\in\LR$, we write
\begin{equation}
  A_k:=A|_{\tau_k\times\sigma_k}
  \approx T_k=U_k\Sigma_kV_k^T,
  \qquad T_k\in\R^{m_k\times n_k},
  \label{eq:h-local-svd}
\end{equation}
with retained rank $r_k$. Geometric inadmissibility does not prevent such a factorization, but no uniformly small rank or storage reduction is assumed for near-field blocks; the retained rank can reach the smaller block dimension. Only diagonal blocks are stored exactly, while each $T_k$ is balanced, quantized, and evaluated independently. Sections~\ref{sec:quant_models} and~\ref{sec:allocation} first establish the local theory; their notation suppresses the block index $k$ for readability.  The index is restored in the global results of Algorithm~\ref{alg:h-construction} and Theorem~\ref{thm:h-algorithm}.

The present representation retains a standard hierarchical block partition but also factorizes its non-diagonal inadmissible blocks. It provides storage accounting and reconstruction error bounds for the blocks, without invoking a geometric low-rank guarantee for the additional blocks.  Nested-basis $\mathcal H^2$ representations and forward-error bounds for hierarchical matrix--matrix multiplication require additional basis-transfer and cross-block accumulation terms and are outside the current scope.

Integer and block-scaled formats offer two complementary benefits.  First, prequantized factors reduce memory traffic, cache pressure, and communication.  Second, on hardware with INT8 or block-scaled matrix instructions, the factor--matrix products arising when a non-diagonal block is applied to an input matrix can be mapped to low-precision matrix kernels with a higher arithmetic throughput.  However, this does not necessarily lead to performance gain since quantization and scale processing introduce overhead, e.g., small off-diagonal blocks can lead to inefficient kernel launches, and a mathematically convenient placement of scales may not map to a single  GEMM supported by hardware.


\subsection{Quantization model}\label{sec:quant_models}
We analyze the error introduced by integer quantization of balanced low-rank factors. Floating-point reconstruction error is treated separately in the GEMM-aware analysis below.

For a matrix $Z$, $\normmax{Z}=\max_{i,j}|Z_{ij}|$.  For a vector $z$, $\norminf{z}=\max_i |z_i|$.  Since the columns of $U$ and $V$ are unit 2-norm vectors,
\begin{equation}
  \alpha_j:=\norminf{u_j}\le 1,
  \qquad
  \beta_j:=\norminf{v_j}\le 1.
  \label{eq:alphabeta}
\end{equation}

The singular value aware motivation follows the same general principle as mixed-precision low-rank representations such as~\cite{10.1093/imanum/drac037}: components associated with smaller singular values may be stored at lower precision because their contribution to the low-rank block is smaller.  The bounds below are derived for the symmetric integer quantizer defined next. 

For a bitwidth $b\ge2$, set  $q_b = 2^{b-1}-1$. Given a nonzero matrix or vector $Z$, define the symmetric scale $s_b(Z)=\normmax{Z} / q_b$ for matrices, with the analogous definition using $\norminf{Z}$ for vectors. The reconstructed quantized object is
\begin{equation}
  \Q_b(Z)=s_b(Z)\,\operatorname{round}\!\left(\frac{Z}{s_b(Z)}\right).
  \label{eq:quantizer}
\end{equation}
For $Z=0$, set $\Q_b(0)=0$.  Because the scale is chosen from the largest magnitude entry, no clipping is needed in exact arithmetic.

\begin{assumption}
\label{ass:quantizer}
Throughout the paper we assume:
\begin{enumerate}[label=(A\arabic*)]
  \item rounding to nearest in~\eqref{eq:quantizer};
  \item the scale factors are stored exactly, or at sufficiently high precision that scale-rounding errors are negligible relative to the modeled data quantization error;
  \item the stated bitwidth $b$ refers to the signed integer quantizer in~\eqref{eq:quantizer};
  \item when column-wise balanced quantization is used, each balanced column has its own scale for the left factor and its own scale for the right factor;
  \item the bit allocation uses the same bitwidth $b_j$ for the left and right balanced columns associated with singular component $j$;
  \item storage formulae below count the quantized payload; fixed scale metadata is discussed separately in Section~\ref{sec:assumptions-discussion}.
\end{enumerate}
\end{assumption}

\begin{lemma}
\label{lem:elementary-quant}
Under Assumption~\ref{ass:quantizer}, for a matrix $Z$ and a vector $z$, we have 
\begin{center}
\begin{tabular}{@{}r@{\quad}l@{\hspace{1.5em}}r@{\quad}l@{}}
    \sideeqtag{eq:matrix-qerr} &
    $\displaystyle
      \normmax{\Q_b(Z)-Z}
      \le \frac{\normmax{Z}}{2q_b}$, &
    \sideeqtag{eq:vector-qerr} &
    $\displaystyle
      \norminf{\Q_b(z)-z}
      \le \frac{\norminf{z}}{2q_b}$.
\end{tabular}
\end{center}

\end{lemma}

\begin{proof}
Each rounded integer differs from its unrounded value by at most $1/2$. Multiplying by the scale $s_b(Z)=\normmax{Z}/q_b$ yields $\normmax{\Q_b(Z)-Z}\le s_b(Z)/2$.  The vector case is identical.
\end{proof}

It is convenient to define the bit-dependent coefficient
\begin{equation}
  c(b)
  := \frac{1}{q_b}+\frac{1}{4q_b^2}.
  \label{eq:cb-def}
\end{equation}
This coefficient will appear repeatedly in the product-error bounds.

\paragraph{Direct quantization of SVD factors}
\label{sec:direct}

We first analyze general factors $X\in\R^{m\times r}$ and $Y\in\R^{n\times r}$ in $T=X\Sigma Y^T$, keeping $\Sigma$ explicit.  Direct quantization of the SVD factors in~\eqref{eq:svd} is the special case $X=U$ and $Y=V$. The resulting error bound depends on $\normmax{X}$ and $\normmax{Y}$. For a matrix $Z$, $\nuc{Z}$ denotes its nuclear norm; in particular, $\nuc{\Sigma}=\sum_{j=1}^{r}\sigma_j$.

\begin{theorem}
\label{thm:direct-general}
Let
\begin{equation}
  T=X\Sigma Y^T,
  \qquad
  \widehat X=X+E,
  \qquad
  \widehat Y=Y+F,
\end{equation}
where $\Sigma=\operatorname{diag}(\sigma_1,\ldots,\sigma_r)$ with $\sigma_j\ge0$.  Supposing $\normmax{E}\le \Delta_x$ and $\normmax{F}\le \Delta_y$, with $\widehat T=\widehat X\Sigma\widehat Y^T$ and $\alpha=\normmax{X}$ and $\beta=\normmax{Y}$, we have
\begin{equation}
  \normmax{T-\widehat T}
  \le
  \left(
    \Delta_x\beta
    +\alpha\Delta_y
    +\Delta_x\Delta_y
  \right)
  \sum_{j=1}^{r}\sigma_j =
  \left(
    \Delta_x\beta
    +\alpha\Delta_y
    +\Delta_x\Delta_y
  \right)
  \nuc{\Sigma}.
  \label{eq:direct-general-bound}
\end{equation}
\end{theorem}

\begin{proof}
Expanding the product gives
\begin{equation}
  \widehat T-T
  =E\Sigma Y^T+X\Sigma F^T+E\Sigma F^T.
\end{equation}
For every entry $(i,k)$,
\begin{align}
  |(E\Sigma Y^T)_{ik}|
  &\le
  \sum_{j=1}^{r}|E_{ij}|\,\sigma_j\,|Y_{kj}|
  \le \Delta_x\beta\sum_{j=1}^{r}\sigma_j,\\
  |(X\Sigma F^T)_{ik}|
  &\le \alpha\Delta_y\sum_{j=1}^{r}\sigma_j, \quad
  |(E\Sigma F^T)_{ik}|
  \le \Delta_x\Delta_y\sum_{j=1}^{r}\sigma_j.
\end{align}
Taking the maximum over all entries and applying the triangle inequality yields \eqref{eq:direct-general-bound}.
\end{proof}

\begin{corollary}
\label{cor:direct-symmetric}
If $X$ and $Y$ are quantized by~\eqref{eq:quantizer} at the same bitwidth $b$, then
\begin{equation}
  \normmax{T-\widehat T}
  \le
  \normmax{X}\,\normmax{Y}\,c(b)\,\nuc{\Sigma}.
  \label{eq:direct-symmetric-bound}
\end{equation}
\end{corollary}

\begin{proof}
By Lemma~\ref{lem:elementary-quant},
\begin{equation}
  \Delta_x=\frac{\alpha}{2q_b},
  \qquad
  \Delta_y=\frac{\beta}{2q_b}.
\end{equation}
Substitution into Theorem~\ref{thm:direct-general} gives
\begin{equation}
  \Delta_x\beta+\alpha\Delta_y+\Delta_x\Delta_y
  =\alpha\beta
   \left(\frac{1}{q_b}+\frac{1}{4q_b^2}\right).
\end{equation}
\end{proof}

\begin{corollary}
\label{cor:direct-svd}
For SVD factors $U$ and $V$, take $X=U$ and $Y=V$ in Corollary~\ref{cor:direct-symmetric}.  Then
\begin{equation}
  \normmax{T-\widehat T}
  \le c(b)\sum_{j=1}^{r}\sigma_j = c(b) \nuc{\Sigma}.
  \label{eq:direct-svd-bound}
\end{equation}
\end{corollary}

\begin{proof}
Every entry of a unit 2-norm vector has magnitude at most one, hence $\normmax{U}\le1$ and $\normmax{V}\le1$.  Apply Corollary~\ref{cor:direct-symmetric}.
\end{proof}

\begin{remark}
A bound of the form
\begin{equation*}
  (\Delta_x+\Delta_y+\Delta_x\Delta_y)\sum_j\sigma_j
\end{equation*}
is valid as a relaxed bound when $\normmax{X}\le1$ and $\normmax{Y}\le1$.  For arbitrary factors, Theorem~\ref{thm:direct-general} retains the max norms of both factors.
\end{remark}

The direct bound uses one scale for each entire factor and does not distinguish the quantization errors of individual singular components.  To assign bitwidths by component, we instead quantize balanced columns separately.

\paragraph{Balanced factors and weighted quantization} 
\label{sec:balanced}

We define balanced factors, followed by a weighted quantization based on singular values. We write the balanced factors as
\begin{equation}
  P:=U\Sigma^{1/2},
  \qquad
  Q:=V\Sigma^{1/2},
  \qquad
  T=PQ^T.
  \label{eq:balanced}
\end{equation}
Thus $p_j=\sqrt{\sigma_j}\,u_j$, $q_j=\sqrt{\sigma_j}\,v_j$, and $T=\sum_{j=1}^{r}p_jq_j^T$. Accordingly, each pair $p_j,q_j$ is quantized with separate scales determined from those columns.

The role of balancing depends on the scaling.  With separate maximum-based column scales and a fixed bit map, reciprocal positive diagonal rescaling $P\mapsto PD$, $Q\mapsto QD^{-1}$ leaves the \emph{quantized product} unchanged.  Indeed, for $d>0$ the column quantizer is positively homogeneous: $\mathcal{Q}_b(dx)=d\mathcal{Q}_b(x)$. Balancing therefore exposes the singular-value weights used in the analysis, and it is not itself a mechanism for reducing ideal quantization error under this model.  A scale shared by several columns does not have this componentwise invariance, and finite-range arithmetic can also break it. Supplementary Section~\ref{sec:supp-scaling} gives the argument and a numerical control.

\begin{lemma}
\label{lem:balanced-rankone}
Given a singular component $j$ and a bitwidth $b_j$.  Quantize $p_j$ and $q_j$ independently with column-wise scales, i.e., $\widehat p_j=p_j+e_j$ and $\widehat q_j=q_j+f_j$, then
\begin{equation}
  \normmax{p_jq_j^T-\widehat p_j\widehat q_j^T}
  \le
  \sigma_j\alpha_j\beta_j\,c(b_j),
  \label{eq:rankone-balanced}
\end{equation}
where $\alpha_j=\norminf{u_j}$ and $\beta_j=\norminf{v_j}$. Consequently,
\begin{equation}
  \normmax{p_jq_j^T-\widehat p_j\widehat q_j^T}
  \le \sigma_j c(b_j).
  \label{eq:rankone-balanced-safe}
\end{equation}
\end{lemma}

\begin{proof}
Since $\norminf{p_j}=\sqrt{\sigma_j}\alpha_j$, $\norminf{q_j}=\sqrt{\sigma_j}\beta_j$, Lemma~\ref{lem:elementary-quant} gives
\begin{equation}
  \norminf{e_j}
  \le\frac{\sqrt{\sigma_j}\alpha_j}{2q_{b_j}},
  \qquad
  \norminf{f_j}
  \le\frac{\sqrt{\sigma_j}\beta_j}{2q_{b_j}}.
\end{equation}
Expanding the rank-one product,
\begin{equation}
  \widehat p_j\widehat q_j^T-p_jq_j^T
  =e_jq_j^T+p_jf_j^T+e_jf_j^T.
\end{equation}
Using $\normmax{xy^T}=\norminf{x}\norminf{y}$, we have
\begin{equation*}
  \normmax{e_jq_j^T}
  \le \frac{\sigma_j\alpha_j\beta_j}{2q_{b_j}},\quad
  \normmax{p_jf_j^T}
  \le \frac{\sigma_j\alpha_j\beta_j}{2q_{b_j}},\quad
  \normmax{e_jf_j^T}
  \le \frac{\sigma_j\alpha_j\beta_j}{4q_{b_j}^2}.
\end{equation*}
Summing gives~\eqref{eq:rankone-balanced}.  Equation \eqref{eq:rankone-balanced-safe} follows from~\eqref{eq:alphabeta}.
\end{proof}

An adaptive assignment may use a different bitwidth for each rank-one term in $T$. The error bound for the complete block must therefore reflect these individual choices rather than a single width for the whole factor. The following theorem bounds the error from quantizing the balanced factors.

\begin{theorem}
\label{thm:balanced-mixed}
Let each singular component $j$ use a possibly different bitwidth $b_j$ and column-wise balanced quantization as in Lemma~\ref{lem:balanced-rankone}. Let $\widehat T :=\sum_{j=1}^{r}\widehat p_j\widehat q_j^T$, then
\begin{equation}
  \normmax{T-\widehat T}
  \le
  \sum_{j=1}^{r}
  \sigma_j\alpha_j\beta_j\,c(b_j).
  \label{eq:balanced-sharp}
\end{equation}
In particular, the singular values alone give the bound
\begin{equation}
  \normmax{T-\widehat T}
  \le
  \sum_{j=1}^{r}\sigma_j c(b_j).
  \label{eq:balanced-safe}
\end{equation}
\end{theorem}

\begin{proof}
Following $T-\widehat T =\sum_{j=1}^{r}  \left(p_jq_j^T-\widehat p_j\widehat q_j^T\right)$, we apply the triangle inequality in max norm and then Lemma~\ref{lem:balanced-rankone}.  The bound follows from $\alpha_j\beta_j\le1$.
\end{proof}

The first bound retains the sizes of the singular-vector entries; the second uses only the singular values.  We use the latter to choose bitwidths because the singular values are ordered, whereas the geometry-dependent weights need not be. 
In the following, we denote the two bounds in Theorem~\ref{thm:balanced-mixed} by $E_{\mathrm{geom}}:=\sum_{j=1}^r \sigma_j\alpha_j\beta_jc(b_j)$ and $E_{\mathrm{sv}}:=\sum_{j=1}^r\sigma_jc(b_j)$, respectively.

\begin{corollary}
\label{cor:channel-bound}
Let $\{G_1,\ldots,G_s\}$ be a disjoint partition of $\{1,\ldots,r\}$ and let every component in group $G_\ell$ use bitwidth $b_\ell$.  Then
\begin{equation}
  \normmax{T-\widehat T}
  \le
  \sum_{\ell=1}^{s} c(b_\ell)
  \sum_{j\in G_\ell}\sigma_j\alpha_j\beta_j
  \le
  \sum_{\ell=1}^{s} c(b_\ell)
  \sum_{j\in G_\ell}\sigma_j.
  \label{eq:channel-bound}
\end{equation}
Hence the error contributed by a group can be bounded using only the sum of its singular values.
\end{corollary}

The geometry-aware weights can reduce conservatism, but they need not inherit the ordering of the singular values.  The monotone singular value allocation theorem in Section~\ref{sec:allocation} therefore uses the weights $\sigma_j$.

The preceding bounds use a separate scale for each column.  If columns assigned the same bitwidth instead share a scale, the error bound changes.

\begin{lemma}
\label{lem:shared-scale}
Let $G\subseteq\{1,\ldots,r\}$ contain $g=|G|$ components and define
\begin{equation}
  P_G=[p_j]_{j\in G}\in\R^{m\times g},
  \qquad
  Q_G=[q_j]_{j\in G}\in\R^{n\times g}.
\end{equation}
Suppose $P_G$ is quantized using one shared scale and $Q_G$ is quantized using one shared scale, both at bitwidth $b$.  Let $a_G=\normmax{P_G}$ and $d_G=\normmax{Q_G}$,  then
\begin{equation}
  \normmax{P_GQ_G^T-\widehat P_G\widehat Q_G^T}
  \le g\,a_Gd_G\,c(b).
  \label{eq:shared-scale-bound}
\end{equation}
\end{lemma}

\begin{proof}
Let $E_G=\widehat P_G-P_G$ and $F_G=\widehat Q_G-Q_G$. Lemma~\ref{lem:elementary-quant} gives
\begin{equation}
  \normmax{E_G}\le\frac{a_G}{2q_b},
  \qquad
  \normmax{F_G}\le\frac{d_G}{2q_b}.
\end{equation}
For matrices $A\in\R^{m\times g}$ and $B\in\R^{n\times g}$, we have
\begin{equation}
  \normmax{AB^T}\le g\normmax{A}\normmax{B}.
\end{equation}
Expanding
\begin{equation}
  \widehat P_G\widehat Q_G^T-P_GQ_G^T
  =E_GQ_G^T+P_GF_G^T+E_GF_G^T
\end{equation}
and applying the preceding inequality yields~\eqref{eq:shared-scale-bound}.
\end{proof}

\begin{corollary}
\label{cor:shared-svd}
For balanced SVD columns,
\begin{equation}
  a_G\le\sqrt{\sigma_{\max}(G)},
  \qquad
  d_G\le\sqrt{\sigma_{\max}(G)},
\end{equation}
where $\sigma_{\max}(G)=\max_{j\in G}\sigma_j$.  Therefore
\begin{equation}
  \normmax{P_GQ_G^T-\widehat P_G\widehat Q_G^T}
  \le g\,\sigma_{\max}(G)c(b).
  \label{eq:shared-svd-safe}
\end{equation}
\end{corollary}

\begin{proof} Each column of $U$ and $V$ has unit 2-norm, so $|u_{ij}|\le 1$ and $|v_{ij}|\le 1$. Thus $a_G,d_G\le\sqrt{\sigma_{\max}(G)}$. By substituting these bounds into~\eqref{eq:shared-scale-bound},  it gives~\eqref{eq:shared-svd-safe}. \end{proof}

\begin{remark}
For column-wise scales, the bound for a group is proportional to $\sum_{j\in G}\sigma_j$.  For one shared scale, the safe bound is proportional to $g\sigma_{\max}(G)$, which can be substantially larger when the singular values decay across the group.  Therefore the additive bound allocation rule $E_G\le c(b)\sum_{j\in G}\sigma_j$ should only be used under the column-wise scaling model (or under another model for which that inequality has been separately proved).
\end{remark}

\paragraph{Floating-point reconstruction of bitwidth groups}
\label{sec:gemm}

The bounds above compare the exact product $T$ with the exact real-arithmetic product of its reconstructed quantized factors.  In an implementation, one GEMM forms the product for each bitwidth group, and these products are summed in floating point. Both operations can introduce rounding. The analysis applies whether the stored integers are explicitly dequantized before the GEMM or an equivalent scaled kernel is used, provided that the returned product obeys the standard accumulation model below.

\begin{assumption}
\label{ass:gemm}
Let $u_{\mathrm a}$ be the unit roundoff of the GEMM accumulator. Here $\fl(\cdot)$ denotes the computed floating-point result. Products and sums do not overflow or underflow, and a length-$g$ dot product satisfies the standard componentwise bound
\begin{equation}
  \abs{\fl(x^Ty)-x^Ty}
  \le \gamma_g(u_{\mathrm a})\,|x|^T|y|,
  \qquad
  \gamma_g(u):=\frac{gu}{1-gu},
  \label{eq:gamma-def}
\end{equation}
whenever $gu<1$.  We set $\gamma_0(u)=0$.
\end{assumption}

\begin{lemma}
\label{lem:channel-gemm}
Let $G_\ell$ contain $g_\ell$ balanced columns, all stored at bitwidth $b_\ell$ but with separate column scales.  Define
\begin{equation}
  \widehat P_\ell=[\widehat p_j]_{j\in G_\ell},
  \qquad
  \widehat Q_\ell=[\widehat q_j]_{j\in G_\ell},
  \qquad
  C_\ell:=\fl(\widehat P_\ell\widehat Q_\ell^T).
\end{equation}
Under Assumptions~\ref{ass:quantizer} and~\ref{ass:gemm},
\begin{equation}
  \normmax{C_\ell-\widehat P_\ell\widehat Q_\ell^T}
  \le \gamma_{g_\ell}(u_{\mathrm a})W_\ell,
  \label{eq:gemm-channel-raw}
\end{equation}
where $W_\ell :=\normmax{|\widehat P_\ell|\,|\widehat Q_\ell|^T}.$
  
Let
\begin{equation}
  \rho(b):=1+\frac{1}{2q_b},
  \qquad
  \Omega_\ell:=\sum_{j\in G_\ell}
  \sigma_j\alpha_j\beta_j,
  \qquad
  S_\ell:=\sum_{j\in G_\ell}\sigma_j.
\end{equation}
Then
\begin{equation}
  W_\ell\le \rho(b_\ell)^2\Omega_\ell
  \le \rho(b_\ell)^2S_\ell.
  \label{eq:gemm-channel-weight}
\end{equation}
Consequently, if $T_{G_\ell}=P_{G_\ell}Q_{G_\ell}^T$, we have
\begin{equation}
  \normmax{T_{G_\ell}-C_\ell}
  \le
  \left[c(b_\ell)
  +\gamma_{g_\ell}(u_{\mathrm a})\rho(b_\ell)^2\right]
  \Omega_\ell.
  \label{eq:one-channel-total}
\end{equation}
Replacing $\Omega_\ell$ by $S_\ell$ gives an error bound depending only on the singular values.
\end{lemma}

\begin{proof}
Applying~\eqref{eq:gamma-def} entrywise to the group's GEMM gives \eqref{eq:gemm-channel-raw}.  Lemma~\ref{lem:elementary-quant} and the balanced column norms give
\begin{equation}
  \norminf{\widehat p_j}
  \le \sqrt{\sigma_j}\alpha_j\rho(b_\ell),
  \qquad
  \norminf{\widehat q_j}
  \le \sqrt{\sigma_j}\beta_j\rho(b_\ell).
\end{equation}
Therefore, for every entry,
\begin{align}
  (|\widehat P_\ell|\,|\widehat Q_\ell|^T)_{ik}
  &\le \sum_{j\in G_\ell}
  \norminf{\widehat p_j}\norminf{\widehat q_j} \le \rho(b_\ell)^2\Omega_\ell,
\end{align}
which proves~\eqref{eq:gemm-channel-weight}.  Finally, split
\begin{equation}
  T_{G_\ell}-C_\ell
  =\left(T_{G_\ell}-\widehat P_\ell\widehat Q_\ell^T\right)
   +\left(\widehat P_\ell\widehat Q_\ell^T-C_\ell\right),
\end{equation}
and combine Corollary~\ref{cor:channel-bound} with \eqref{eq:gemm-channel-raw}.
\end{proof}
Using the above lemma, we can now state a bound on the local reconstruction error. 
\begin{theorem}
\label{thm:local-gemm}
Let $G_1,\ldots,G_s$ be the bitwidth groups and let every $C_\ell$ be formed as in Lemma~\ref{lem:channel-gemm}.  Suppose their products are accumulated sequentially in the same precision,
\begin{equation}
  \widetilde T:=\fl\!\left(\sum_{\ell=1}^{s}C_\ell\right).
\end{equation}
Then
\begin{equation}
\begin{aligned}
  \normmax{T-\widetilde T}
  \le \sum_{\ell=1}^{s}
  \Bigl[&c(b_\ell)
  +\gamma_{g_\ell}(u_{\mathrm a})\rho(b_\ell)^2\\
  &+\gamma_{s-1}(u_{\mathrm a})
   \bigl(1+\gamma_{g_\ell}(u_{\mathrm a})\bigr)
   \rho(b_\ell)^2\Bigr]\Omega_\ell.
  \label{eq:local-gemm-total}
\end{aligned}
\end{equation}
The same expression with $\Omega_\ell$ replaced by $S_\ell$ is a conservative singular-value-only error bound.
\end{theorem}

\begin{proof}
The first two terms account for factor quantization and rounding in each GEMM, as in Lemma~\ref{lem:channel-gemm}. Standard sequential summation gives
\begin{equation}
  \normmax{\fl \left(\sum_\ell C_\ell \right)-\sum_\ell C_\ell}
  \le \gamma_{s-1}(u_{\mathrm a})\sum_{\ell=1}^{s}\normmax{C_\ell}.
\end{equation}
Moreover,~\eqref{eq:gemm-channel-raw} implies
\begin{equation}
  \normmax{C_\ell}
  \le \bigl(1+\gamma_{g_\ell}(u_{\mathrm a})\bigr)W_\ell.
\end{equation}
Insert~\eqref{eq:gemm-channel-weight} and add the three error contributions.
\end{proof}

The coefficient $c(b_\ell)$ is caused by quantizing the two balanced factors; $\gamma_{g_\ell}\rho(b_\ell)^2$ is the rounding error inside the GEMM for group $G_\ell$; and the final term is caused by adding the $s$ group products.  Higher accumulation precision reduces the last two terms without changing the stored factor bits.  If the integer dot products are accumulated exactly, the GEMM term is zero for that stage, but rounding in scale application and summation of group products must still be modeled in the precision where those operations are performed. In addition, pairwise or compensated summation may replace the sequential-summation factor $\gamma_{s-1}$ by the corresponding sharper bound.

\subsection{Singular value aware bit allocation}
\label{sec:allocation}

The error bound for a reconstructed block has three contributions: low-rank approximation, factor quantization, and reconstruction arithmetic. Bit allocation uses the part of the block budget left after the other two contributions have been accounted for.

\begin{lemma}
\label{lem:trunc-quant}
Let $A_k$ be a matrix block, let $T_k$ be its low-rank approximation, and let $\widehat T_k$ be the exact real-arithmetic product of its reconstructed quantized factors.  Let $\widetilde T_k$ denote the matrix actually returned by the GEMMs for the bitwidth groups and their summation.  Let $\eta_k$ be any upper bound satisfying $\eta_k\ge\normmax{A_k-T_k}$, and define
\begin{equation}
  Q_k:=\normmax{T_k-\widehat T_k},
  \qquad
  R_k:=\normmax{\widehat T_k-\widetilde T_k}.
\end{equation}
Then
\begin{equation}
  \normmax{A_k-\widetilde T_k}
  \le \eta_k+Q_k+R_k.
  \label{eq:block-total-error}
\end{equation}
\end{lemma}

\begin{proof}
Insert $T_k$ and $\widehat T_k$ between $A_k$ and $\widetilde T_k$ and apply the triangle inequality.
\end{proof}

For a bit assignment $b_{k,1},\ldots,b_{k,r_k}$, define the quantization error bound
\begin{equation}
  Q_k^{\mathrm{cert}}
  :=\sum_{j=1}^{r_k}\sigma_{k,j}c(b_{k,j}),
  \label{eq:qk-cert}
\end{equation}
so that $Q_k\le Q_k^{\mathrm{cert}}$ by Theorem~\ref{thm:balanced-mixed}.  Let $R_k^{\mathrm{cert}}$ be the sum of the two GEMM-related terms in~\eqref{eq:local-gemm-total}.  In the conservative singular-value-only form, if block $k$ has $s_k$ nonempty bitwidth groups, this is
\begin{equation}
\begin{aligned}
  R_k^{\mathrm{cert}}
  :=\sum_{\ell=1}^{s_k}
  \Bigl[&\gamma_{g_{k,\ell}}(u_{\mathrm a})
          \rho(b_{k,\ell})^2+\gamma_{s_k-1}(u_{\mathrm a})
   \bigl(1+\gamma_{g_{k,\ell}}(u_{\mathrm a})\bigr)
   \rho(b_{k,\ell})^2\Bigr]
   \sum_{j\in G_{k,\ell}}\sigma_{k,j}.
  \label{eq:rk-cert}
\end{aligned}
\end{equation}
Then $R_k\le R_k^{\mathrm{cert}}$.

\begin{remark}
\label{rem:quant-budget}
Suppose the desired total block error constraint is
\begin{equation}
  \normmax{A_k-\widetilde T_k}
  \le B_k^{\mathrm{tot}}.
\end{equation}
Choose a nonnegative allowance $B_k^{\mathrm{gemm}}$ for GEMM rounding and define
\begin{equation}
  B_k^{\mathrm{q}}
  :=B_k^{\mathrm{tot}}-\eta_k-B_k^{\mathrm{gemm}}.
  \label{eq:quant-budget}
\end{equation}
If $B_k^{\mathrm q}\ge0$, it is sufficient to impose
\begin{equation}
  Q_k^{\mathrm{cert}}\le B_k^{\mathrm q},
  \qquad
  R_k^{\mathrm{cert}}\le B_k^{\mathrm{gemm}}.
  \label{eq:budget-two-tests}
\end{equation}
If $\eta_k+B_k^{\mathrm{gemm}}>B_k^{\mathrm{tot}}$, the low-rank approximation or the GEMM realization must first be improved.  In exact arithmetic, $B_k^{\mathrm{gemm}}=0$ recovers the quantization-only budget $B_k^{\mathrm q}=B_k^{\mathrm{tot}}-\eta_k$.
\end{remark}

Each entry of $A$ belongs to exactly one block of the partition. Consequently, the global max-norm error is the largest of the blockwise max-norm errors. We present the global reconstruction error in the following theorem.

\begin{theorem}
\label{thm:h-global}
Let the blocks satisfy~\eqref{eq:leaf-partition}.  Store every diagonal block exactly and, for each non-diagonal block $k\in\LR$, including the inadmissible ones, construct $\widetilde T_k$ as above.  Let $\widetilde A$ be the matrix obtained by placing these block representations in their original positions, and write $\widetilde A_k$ for its restriction to block $k$.  Define the maximum over an empty $\LR$ as zero.  Then
\begin{equation}
  \normmax{A-\widetilde A}
  =\max_{k\in\mathcal L}\normmax{A_k-\widetilde A_k}
  \le
  \max_{k\in\LR}
  \left(\eta_k+Q_k^{\mathrm{cert}}+R_k^{\mathrm{cert}}\right).
  \label{eq:h-global-bound}
\end{equation}
In particular, if $\eta_k+Q_k^{\mathrm{cert}}+R_k^{\mathrm{cert}} \le B_k^{\mathrm{tot}}$ for every factorized block, then
\begin{equation}
  \normmax{A-\widetilde A}
  \le \max_{k\in\LR}B_k^{\mathrm{tot}}.
  \label{eq:h-global-budget}
\end{equation}
\end{theorem}

\begin{proof}
Because the blocks do not overlap, each entry of $A-\widetilde A$ belongs to exactly one block.  Its global maximum magnitude is therefore the maximum of the block-wise max norms.  Exact diagonal blocks contribute zero, and Lemma~\ref{lem:trunc-quant} together with the two local error bounds gives the inequality.
\end{proof}

The bound also controls the exact application of the stored representation. Let $\widehat A$ contain the exact products of the reconstructed quantized factors, with diagonal blocks stored exactly. Then
\begin{equation*}
\|A-\widehat A\|_{\max}
\le E_{\mathrm{rep}}
:=\max_{k\in\LR}(\eta_k+Q_k^{\mathrm{cert}}),
\end{equation*}
with the maximum taken as zero when $\LR$ is empty, i.e., $E_{\mathrm{rep}} =0$ when $\LR=\varnothing$.

For $X\in\mathbb R^{N\times p}$, it follows that
\begin{equation*}
 \|AX-\widehat A X\|_{\max}
 \le E_{\mathrm{rep}}
       \max_{1\le\ell\le p}\sum_{j=1}^{N}|X_{j\ell}|.
\end{equation*}
Indeed, each entry of $(A-\widehat A)X$ is a sum of products, each bounded by $E_{\mathrm{rep}}|X_{j\ell}|$. The computed application also incurs rounding in the factor--matrix products and block accumulation, and those errors are considered separately in supplementary Section~\ref{sec:supp-application}.

\paragraph{Choosing block-wise error bounds}

A direct global-relative choice is
\begin{equation}
  B_k^{\mathrm{tot}}
  =\varepsilon_{\mathrm{tot}}\normmax{A}, \qquad k\in\LR.
  \label{eq:block-global-budget}
\end{equation}
By Theorem~\ref{thm:h-global}, this yields $\normmax{A-\widetilde A}\le \varepsilon_{\mathrm{tot}}\normmax{A}$.  A locally scaled alternative is
\begin{equation}
  B_k^{\mathrm{tot}}
  =\varepsilon_{\mathrm{tot}}\normmax{A_k}.
  \label{eq:block-relative-budget}
\end{equation}
Since $\normmax{A_k}\le\normmax{A}$, the same global relative guarantee follows, while every block is controlled relative to its own scale.  For very small far-field blocks, a less restrictive choice of block budget is
\begin{equation}
  B_k^{\mathrm{tot}}
  =\varepsilon_{\mathrm{tot}}
   \max\!\left\{\normmax{A_k},\gamma\normmax{A}\right\},
  \qquad 0<\gamma\le1,
  \label{eq:hybrid-budget}
\end{equation}
which also remains no larger than $\varepsilon_{\mathrm{tot}}\normmax{A}$. When $A$ is not explicitly formed, as in some kernel- or matvec-based hierarchical matrix constructions, the global-relative and hybrid choices require $\normmax{A}$, whereas the block-relative choice uses the norms of the individual blocks. Replacing a norm by an upper bound makes the resulting guarantee relative to that bound.

If only the low-rank residual bound $\eta_k\le\varepsilon_{\mathrm{lr},k}\normmax{A}$ is known, the global-relative block budget gives the conservative quantization budget
\begin{equation}
  B_k^{\mathrm{q}}
  =\left(\varepsilon_{\mathrm{tot}}-\varepsilon_{\mathrm{lr},k}\right)
   \normmax{A}-B_k^{\mathrm{gemm}},
  \label{eq:relative-split-budget}
\end{equation}
which is valid when the right-hand side is nonnegative.

\begin{remark}
The separate bounds give a sufficient condition: the allocator first enforces $Q_k^{\mathrm{cert}}\le B_k^{\mathrm q}$ and the resulting groups are then checked against $R_k^{\mathrm{cert}}\le B_k^{\mathrm{gemm}}$.  Since $R_k^{\mathrm{cert}}$ depends on the final group sizes and bitwidths, a less conservative implementation could instead allocate against the joint condition
\begin{equation*}
  Q_k^{\mathrm{cert}}+R_k^{\mathrm{cert}}
  \le B_k^{\mathrm{tot}}-\eta_k.
\end{equation*}
This couples storage selection with the arithmetic realization and is therefore left outside the present greedy allocation analysis.
\end{remark}

\begin{remark}
In \eqref{eq:h-global-bound}, the error contributions are added within each block, while the global entrywise max-norm bound takes the maximum over blocks.  Thus, for the entrywise max-norm representation error, a global budget need not be divided by $|\LR|$.  Such a division would be unnecessarily conservative.  This conclusion depends on the disjoint block partition and does not apply unchanged to a norm or operation in which several block contributions are summed into the same output entry.
\end{remark}

\begin{remark}
The global theorem assumes exact diagonal blocks relative to the reference matrix.  If these blocks are converted to a lower working precision or quantized, and if $D_k$ bounds the resulting error on each diagonal block, then the right-hand side of~\eqref{eq:h-global-bound} is replaced by
\begin{equation*}
  \max\left\{
  \max_{k\in\LR}\left(\eta_k+Q_k^{\mathrm{cert}}+R_k^{\mathrm{cert}}\right),
  \max_{k\in\Dense}D_k
  \right\}.
\end{equation*}
\end{remark}

The geometry-dependent bound can be sharper, but its weights need not follow the order of the singular values.  For the allocation, we therefore use the bound
\begin{equation}
  E_{\mathrm{sv}}(b_1,\ldots,b_r)
  :=\sum_{j=1}^{r}\sigma_jc(b_j)
  \label{eq:allocation-error}
\end{equation}
to select the precision of each singular component.  For a hierarchical block, the available budget $B^{\mathrm q}$ is the amount left after the low-rank error and the allowance for GEMM rounding are subtracted from the block budget, as in~\eqref{eq:quant-budget}; the block index is again suppressed in this section.

Let the available bitwidths be
\begin{equation}
  \Bset=\{\mathfrak b_1,\ldots,\mathfrak b_L\},
  \qquad
  2\le\mathfrak b_1<\mathfrak b_2<\cdots<\mathfrak b_L.
  \label{eq:bitset}
\end{equation}
Set $c_\ell:=c(\mathfrak b_\ell)$.  The following monotonicity observation determines the smallest possible value of $E_{\mathrm{sv}}$.

\begin{lemma}
\label{lem:c-monotone}
The coefficient $c(b)$ in~\eqref{eq:cb-def} strictly decreases with the bit width $b\ge2$. Hence $c_1>c_2>\cdots>c_L>0$.
\end{lemma}

\begin{proof}
$q_b=2^{b-1}-1$ is strictly increasing in $b$.  For $q>0$, the function $f(q)=1/q+1/4q^2$ has derivative
\begin{equation}
  f'(q)=-\frac{1}{q^2}-\frac{1}{2q^3}<0.
\end{equation}
Thus $c(b)=f(q_b)$ strictly decreases as $b$ increases.
\end{proof}

\paragraph{Bitwidth allocation}

Ignoring fixed scale metadata, a balanced rank-one component stored at $b_j$ bits requires $(m+n)b_j$ bits for its factor entries.  Therefore the allocation problem subject to the error bound is
\begin{equation}
\underset{b_j\in\Bset}{\min}\quad
M_{\mathrm{payload}}
=(m+n)\sum_{j=1}^{r}b_j,
\qquad
\text{s.t.}\quad
\sum_{j=1}^{r}\sigma_j c(b_j)\le B^{\mathrm q}.
\label{eq:allocation-problem}
\end{equation}

As such, the role of singular values in storage  is obvious. Let $\varepsilon-j$ be the positive error allowances, which is assigned to the invidual components, with $\sum_{j=1}^r\varepsilon_j\le B^{\mathrm q}$. Therefore, the smallest available bitwidth meeting the allowance for component $j$ is $$b_j^{\mathrm{req}} =\min\{b\in\Bset:\sigma_j c(b)\le\varepsilon_j\},$$ which is assumed to be nonempty. Since $c(b)\sim 2^{1-b}$ as $b$ increases, this gives the approximate relation $b_j^{\mathrm{req}}\approx \lceil 1+\log_2(\sigma_j/\varepsilon_j)\rceil$ before restriction to the available bitwidths. \eqref{eq:allocation-problem} chooses the bitwidths under a total error budget without prescribing the individual allowances $\varepsilon_j$.

\begin{corollary}
\label{cor:feasibility}
The constraint $E_{\mathrm{sv}}\le B^{\mathrm q}$ in problem~\eqref{eq:allocation-problem} is feasible if and only if
\begin{equation}
  c_L\sum_{j=1}^{r}\sigma_j\le B^{\mathrm q}.
  \label{eq:feasibility}
\end{equation}
\end{corollary}

\begin{proof}
The smallest available error coefficient is $c_L$, so assigning every component to $\mathfrak b_L$ gives the minimum value of $E_{\mathrm{sv}}$ among all admissible bit assignments.  Thus~\eqref{eq:feasibility} is both necessary and sufficient.
\end{proof}

Swapping the bitwidths of two components leaves their storage cost unchanged but can change their contribution to the error bound. This suggests assigning no fewer bits to a larger singular value.

\begin{theorem}  
\label{thm:exchange}
Assume $\sigma_1\ge\sigma_2\ge\cdots\ge\sigma_r$.  For the optimization problem~\eqref{eq:allocation-problem}, there exists a storage-optimal feasible assignment satisfying $b_1\ge b_2\ge\cdots\ge b_r$. Thus a larger singular value never needs fewer bits than a smaller singular value in at least one optimal solution.
\end{theorem}

\begin{proof}
Consider an inversion: indices $i<j$ with $\sigma_i\ge\sigma_j$ but $b_i<b_j$.  Since $c(b)$ decreases with $b$, $c(b_i)>c(b_j)$.  Swapping the two bitwidths leaves storage unchanged.  The change in $E_{\mathrm{sv}}$ is
\begin{align}
  &\sigma_i c(b_j)+\sigma_j c(b_i)
  -\sigma_i c(b_i)-\sigma_j c(b_j) 
  =(\sigma_i-\sigma_j)\bigl(c(b_j)-c(b_i)\bigr)
  \le0.
\end{align}
Hence removing an inversion cannot increase the error bound.  Repeating this exchange produces a monotone assignment with the same storage and no larger error.  Applying this to any storage-optimal feasible assignment proves the claim.
\end{proof}

The exchange argument does not determine how many components receive each bitwidth.  It does show that one optimal assignment can be described by boundaries in the ordered singular value list.

\begin{corollary}
\label{cor:contiguous}
Under the assumptions of Theorem~\ref{thm:exchange}, an optimal assignment can be represented by contiguous index intervals
\begin{equation}
  G_1,G_2,\ldots,G_L,
\end{equation}
where $G_1$ contains a suffix of the smallest singular values assigned to the lowest precision $\mathfrak b_1$, the next nonempty group contains the next suffix, and the dominant singular values are assigned to the highest precisions. Empty groups are allowed.
\end{corollary}

\paragraph{Greedy bit allocation}
\label{sec:greedy}
The greedy strategy assigns the lowest available bitwidth to as many of the smallest singular components as possible, while accounting for the error still required by the unassigned components. At every step, it reserves the \emph{minimum possible future error} of all still-unassigned dominant singular components.  This avoids the failure mode in which too much budget is spent on small singular values and the remaining dominant components cannot fit even at the highest available precision. Algorithm~\ref{alg:certified-greedy} implements this rule.

\begin{algorithm}[htbp]
\small
\caption{Greedy allocation with a bound on the remaining error}
\label{alg:certified-greedy}
\begin{algorithmic}[1]


\ENSURE Feasibility flag; groups $G_1,\ldots,G_L$; precision map $b_1,\ldots,b_r$; quantization error bound $E_{\mathrm{cert}}$; bits used by the factors $M_{\mathrm{payload}}$.

\STATE Compute $c_\ell=c(\mathfrak b_\ell)$ for $\ell=1,\ldots,L$.
\STATE Compute prefix sums $S_t=\sum_{j=1}^{t}\sigma_j$, with $S_0=0$.

\IF{$c_L S_r>B^{\mathrm q}$}
    \RETURN \textsc{infeasible}
\ENDIF

\STATE Set $R\gets B^{\mathrm q}$, $t\gets r$, and $G_\ell\gets\varnothing$ for all $\ell$.

\FOR{$\ell=1,\ldots,L-1$}
    \STATE $C\gets0$
    \WHILE{$t\ge1$ and $C+c_\ell\sigma_t+c_LS_{t-1}\le R$}
        \STATE Prepend $t$ to $G_\ell$.
        \STATE $b_t\gets\mathfrak b_\ell$; $C\gets C+c_\ell\sigma_t$; $t\gets t-1$.
    \ENDWHILE
    \STATE $R\gets R-C$.
\ENDFOR

\IF{$t\ge1$}
    \STATE $G_L\gets\{1,\ldots,t\}$.
    \FOR{$j=1,\ldots,t$}
        \STATE $b_j\gets\mathfrak b_L$.
    \ENDFOR
    \STATE $R\gets R-c_LS_t$.
\ENDIF

\STATE $E_{\mathrm{cert}}\gets B^{\mathrm q}-R$,  $M_{\mathrm{payload}}\gets (m+n)\sum_{j=1}^{r}b_j$.

\RETURN \textsc{feasible}, $\{G_\ell\}$, $\{b_j\}$, $E_{\mathrm{cert}}$, $M_{\mathrm{payload}}$.

\end{algorithmic}
\end{algorithm}

Now we define prefix sums
\begin{equation}
  S_t=\sum_{j=1}^{t}\sigma_j,
  \qquad S_0=0.
\end{equation}
At a stage where indices $1,\ldots,t$ remain unassigned, their minimum possible contribution to $E_{\mathrm{sv}}$ is $c_LS_t$.

\begin{lemma}
\label{lem:greedy-invariant}
At the beginning of each precision level in Algorithm~\ref{alg:certified-greedy}, if indices $1,\ldots,t$ remain unassigned and $R$ is the remaining budget, then
\begin{equation}
  c_LS_t\le R.
  \label{eq:reserve-invariant}
\end{equation}
\end{lemma}

\begin{proof}
The invariant holds initially by the feasibility test.  During a precision level, suppose the current group contributes $C$ to the error bound.  Every accepted index satisfies
\begin{equation}
  C+c_\ell\sigma_t+c_LS_{t-1}\le R.
\end{equation}
After accepting it and updating $C$ and $t$, this becomes $C+c_LS_t\le R$.

When the group is closed, the algorithm sets $R\leftarrow R-C$, yielding $c_LS_t\le R$.  If no index is accepted, neither $R$ nor $t$ changes.  Thus the invariant is preserved.
\end{proof}

Lemma~\ref{lem:greedy-invariant} shows that each completed group leaves enough budget to assign the remaining components at the highest available bitwidth. Theorem~\ref{thm:greedy-feasible} then establishes that the resulting monotone allocation satisfies the quantization budget and the corresponding max-norm error bound.
\begin{theorem}
\label{thm:greedy-feasible}
If~\eqref{eq:feasibility} holds, Algorithm~\ref{alg:certified-greedy} terminates with a monotone bitwidth assignment satisfying
\begin{equation}
  \sum_{j=1}^{r}\sigma_jc(b_j)
  =E_{\mathrm{cert}}
  \le B^{\mathrm q}.
  \label{eq:greedy-cert}
\end{equation}
Consequently, under the column-wise balanced quantization model, we have
\begin{equation}
  \normmax{T-\widehat T}
  \le E_{\mathrm{cert}}
  \le B^{\mathrm q}.
\end{equation}
\end{theorem}

\begin{proof}
By Lemma~\ref{lem:greedy-invariant}, when the algorithm reaches the highest precision, the remaining prefix satisfies $c_LS_t\le R$ and therefore fits in the remaining budget.  All committed group errors are nonnegative, so the final remaining budget is nonnegative.  Hence $E_{\mathrm{cert}}=B^{\mathrm q}-R\le B^{\mathrm q}$.  The assignment is monotone because the algorithm processes singular values from the smallest upward while processing bitwidths from the lowest upward.  The max-norm statement follows from Theorem~\ref{thm:balanced-mixed}.
\end{proof}

\begin{remark}
Algorithm~\ref{alg:certified-greedy} greedily makes each low-precision tail group as large as possible while leaving enough budget for the remaining prefix.  It guarantees feasibility under~\eqref{eq:feasibility}, but does not generally minimize the stored factor bits over the available bitwidths.
\end{remark}

\paragraph{\texorpdfstring{Mixed-bit $\mathcal{H}$-matrix construction}{Hierarchical construction}}
\label{sec:h-algorithm}

The local allocator becomes a hierarchical algorithm by applying it independently to every non-diagonal block, including non-diagonal inadmissible blocks. Here, no allocation is performed before the singular components are known. First, the rank components receive bitwidths, and only then do components with a common bitwidth form a group.  The groups do not share a quantization scale; each column retains its own left and right scales. The admissibility test determines where subdivision stops; both admissible and non-diagonal inadmissible blocks use factorized storage.  If the permitted rank, available bitwidths, and arithmetic choices cannot meet a block-wise error bound, construction reports infeasibility rather than storing that block in dense form.

\begin{algorithm}[htbp]
\small
\caption{Mixed-bit $\mathcal{H}$-matrix construction with diagonal-only dense storage}
\label{alg:h-construction}
\begin{algorithmic}[1]

\REQUIRE Matrix access or block data; hierarchical block partition $\mathcal L=\LR\dot\cup\Dense$; block budgets $B_k^{\mathrm{tot}}$; bit set $\Bset$; GEMM accumulator precision $u_{\mathrm a}$ and reserves $B_k^{\mathrm{gemm}}$.

\FORALL{$k\in\Dense$}
    \STATE Store the diagonal block $A_k$ exactly.
\ENDFOR

\FORALL{non-diagonal $k\in\LR$ (admissible or inadmissible)}
    \STATE Construct or recompress $T_k=U_k\Sigma_kV_k^T$ and obtain $\eta_k\ge\normmax{A_k-T_k}$.

    \STATE Set $B_k^{\mathrm q}\gets B_k^{\mathrm{tot}}-\eta_k-B_k^{\mathrm{gemm}}$.

    \IF{$B_k^{\mathrm q}<0$}
        \STATE Refine the low-rank approximation or increase GEMM accuracy; recompute the budget, or return infeasibility if permitted refinements are exhausted. Continue only after the budget is nonnegative.
    \ENDIF

    \STATE Apply Algorithm~\ref{alg:certified-greedy} to $\{\sigma_{k,j}\}_{j=1}^{r_k}$ with budget $B_k^{\mathrm q}$.

    \IF{the highest-precision feasibility test fails}
        \STATE Increase the admissible precision set or improve the low-rank approximation and recompute the budget; repeat the block allocation, or return infeasibility if permitted refinements are exhausted.
    \ENDIF

    \STATE Form $p_{k,j}=\sqrt{\sigma_{k,j}}u_{k,j}$ and $q_{k,j}=\sqrt{\sigma_{k,j}}v_{k,j}$; quantize each column with its own left and right scale at bitwidth $b_{k,j}$.

\STATE Form groups $G_{k,\ell}=\{j:b_{k,j}=\mathfrak b_\ell\}$ and compute $R_k^{\mathrm{cert}}$ from Theorem~\ref{thm:local-gemm}.

    \IF{$R_k^{\mathrm{cert}}>B_k^{\mathrm{gemm}}$}
        \STATE Increase accumulator precision, use a sharper group summation, or enlarge the rounding allowance and rerun the allocation; return infeasibility if permitted refinements are exhausted.
    \ENDIF

    \STATE Set $E_k\gets \eta_k+Q_k^{\mathrm{cert}}+R_k^{\mathrm{cert}}$.
\ENDFOR

\STATE $E_{\mathcal H}\gets\max_{k\in\LR}E_k$ (zero if $\LR=\varnothing$).

\RETURN $\widetilde A$ and $E_{\mathcal H}$.

\end{algorithmic}
\end{algorithm}

\begin{theorem}
\label{thm:h-algorithm}
If every repeated block step in Algorithm~\ref{alg:h-construction} terminates with $Q_k^{\mathrm{cert}}\le B_k^{\mathrm q}, R_k^{\mathrm{cert}}\le B_k^{\mathrm{gemm}}$, then
\begin{equation}
  \normmax{A-\widetilde A}
  \le E_{\mathcal H}
  \le\max_{k\in\LR}B_k^{\mathrm{tot}}.
\end{equation}
\end{theorem}

\begin{proof}
For each factorized block, Remark~\ref{rem:quant-budget} gives $E_k\le B_k^{\mathrm{tot}}$.  The result then follows directly from Theorem~\ref{thm:h-global}.
\end{proof}

\paragraph{Extension to ACA-generated blocks}
\label{subsec:aca}

The preceding theory does not require the initial low-rank approximation to be constructed by a large dense SVD.  Now suppose adaptive cross approximation (ACA) produces
\begin{equation}
  A_k\approx X_kY_k^T,
  \qquad
  X_k\in\R^{m_k\times r},
  \quad
  Y_k\in\R^{n_k\times r}.
\end{equation}
A small-factor recompression supplies the orthogonal SVD coordinates needed by the singular value aware allocation.  Compute thin QR factorizations $X_k=Q_XR_X$ and $Y_k=Q_YR_Y$, and the SVD of the $R_XR_Y^T \in \mathbb{R}^{r\times r}$, which gives
\begin{equation}
  R_XR_Y^T=\widetilde U\Sigma_k\widetilde V^T \Rightarrow X_kY_k^T
  =(Q_X\widetilde U)\Sigma_k(Q_Y\widetilde V)^T,
\end{equation}
so $U_k=Q_X\widetilde U$ and $V_k=Q_Y\widetilde V$ can be passed directly to the balanced quantization and allocation stages.  Only the small core requires an SVD.

The reconstruction analysis requires a valid low-rank residual bound.  If ACA provides $\eta_k^{\mathrm{ACA}}\ge \normmax{A_k-X_kY_k^T}$ and recompression introduces $\eta_k^{\mathrm{rec}}\ge \normmax{X_kY_k^T-T_k}$, then
\begin{equation}
  \normmax{A_k-T_k}
  \le \eta_k^{\mathrm{ACA}}+\eta_k^{\mathrm{rec}}.
  \label{eq:aca-total}
\end{equation}
Thus one may take $\eta_k:=\eta_k^{\mathrm{ACA}}+\eta_k^{\mathrm{rec}}$ in~\eqref{eq:quant-budget}.  An ACA stopping estimator that is not a proven upper bound gives an empirical rather than a fully certified global result; the quantization and GEMM parts of the analysis remain valid conditionally on the reported $\eta_k$.

\section{Storage accounting}
\label{sec:assumptions-discussion}

The objective in~\eqref{eq:allocation-problem} counts only bits used by the quantized factor entries, omitting scales and bitwidth labels.  Over the complete hierarchical partition, this contribution is
\begin{equation}
  M_{\mathcal H,\mathrm{lr}}
  =\sum_{k\in\LR}(m_k+n_k)
   \sum_{j=1}^{r_k}b_{k,j}.
  \label{eq:h-lr-payload}
\end{equation}
Let $b_{\mathrm d}$ be the number of bits per dense entry. Diagonal blocks add $M_{\mathcal H,\mathrm d} =b_{\mathrm d}\sum_{k\in\Dense}m_kn_k$ bits.  The sum in~\eqref{eq:h-lr-payload} includes non-diagonal inadmissible blocks; a failed non-diagonal candidate makes the strict construction infeasible rather than reverting to dense storage. Each retained component requires two reconstruction scales, so scales stored at $b_s$ bits per scale contribute $M_{\mathrm{scale}}=2b_s\sum_{k\in\LR}r_k$. Including bitwidth labels gives the reported storage model
\begin{equation}
  M_{\mathcal H}
  =M_{\mathcal H,\mathrm{lr}}+M_{\mathcal H,\mathrm d}
   +M_{\mathrm{scale}}+M_{\mathrm{map}},
  \label{eq:total-storage-with-metadata}
\end{equation}
where $M_{\mathrm{map}}$ counts the selected widths.  Hierarchical storage totals include scales and direct precision labels but exclude tree topology, headers, and workspace;~\eqref{eq:allocation-problem} excludes all metadata.  Encoding details are given in supplementary Section~\ref{sec:supp-setup}.

The width $b$ denotes the signed integer quantizer in~\eqref{eq:quantizer}.  Assumption~\ref{ass:quantizer} requires exact or sufficiently accurate scales.  Lower-precision scales introduce errors not included in $E_{\mathrm{geom}}$ or $E_{\mathrm{cert}}$ and require a separate allowance. Theorems~\ref{thm:balanced-mixed} and~\ref{thm:greedy-feasible} use separate left and right scales for each balanced column.  If the columns of a bitwidth group share one scale within each factor, the error bound depends on the group maxima $a_G$ and $d_G$ as in~\eqref{eq:shared-scale-bound}.  Finally, zero singular value components may be discarded before allocation, so $r$ denotes the retained positive rank.

\section{Numerical experiments}\label{sec:exps}
We compare measured errors with the theoretical bounds and evaluate the discrete allocation quality and the storage benefit of quantization in complete hierarchical representations.  The principal storage and bound comparisons use FP64 emulation. Separate C++/CUDA tests examine numerical behavior of the native implementation. On a single NVIDIA A100 GPU, all 1,501 native reconstruction tests with strict arithmetic error bounds satisfied their stated allowances. The mixed-bit path decodes factors to FP32 and uses floating-point GEMMs (see supplementary Section~\ref{sec:supp-native}).

\subsection{Test problems and methodology}\label{sec:exp-setup}

The local tests use $n\times n$ matrices $T=U\Sigma V^T$, with $n\in\{256,512\}$ and rank $r=64$.  We construct the orthonormal columns of $U$ and $V$ by QR factorization of independent Gaussian matrices and prescribe the singular values $\sigma_j=0.82^{j-1}$, $\sigma_j=j^{-1.4}$, or $\sigma_j=1$.  These three cases are referred to below as the Exponential, Algebraic, and Flat spectra.  We use $\Bset=\{4,8,12,16\}$ and six relative tolerances from $0.002$ to $0.1$ (listed in supplementary Section~\ref{sec:supp-setup}), with local quantization budget $B^{\mathrm q}=\varepsilon\normmax{T}$.

The hierarchical suite contains 44 matrices: Euclidean point kernels, BEM matrices, and six supplementary one-dimensional kernels. For the point kernels, $A_{ij}=1/\|x_i-x_j\|_2$ or $\log\|x_i-x_j\|_2$ for $i\ne j$, with $N\in\{512,1024,2048\}$ uniform or clustered points in two or three dimensions. We set $A_{ii}=0$ without regularizing off-diagonal distances; no positive definiteness or three dimensional Laplace interpretation of $\log r$ is assumed.

The BEM matrices discretize single-layer operators by constant panels and midpoint collocation:
\begin{equation*}
 A^{(2)}_{ij}=-\frac{1}{2\pi}\int_{\Gamma_j}\log\|x_i-y\|_2\,ds_y,
 \qquad
 A^{(3)}_{ij}=\frac{1}{4\pi}\int_{\Gamma_j}\frac{dS_y}{\|x_i-y\|_2}.
\end{equation*}
Circle and ellipse boundaries have up to 2048 panels; cube and cuboid boundaries have up to 1536.  Analytic panel integrals include singular self terms, without forced symmetrization.  Supplementary Section~\ref{sec:supp-setup} specifies the geometries and remaining kernels.

We build the common row/column cluster tree by a stable median split along the coordinate of largest spread, stopping when a cluster contains at most 32 indices. For BEM matrices, cluster boxes enclose full panel supports. We use a geometric admissibility test based on cluster bounding boxes~\cite{hackbusch2015hierarchical,bebendorf2008hierarchical}. For boxes $B_\tau$ and $B_\sigma$, a block is admissible when $\operatorname{dist}(B_\tau,B_\sigma)\ge\eta_{\mathrm{adm}}\max\{\operatorname{diam}(B_\tau),\operatorname{diam}(B_\sigma)\}$. Here, distance is the Euclidean separation of the boxes, diameter is the length of a box diagonal, and we set $\eta_{\mathrm{adm}}=1$. Otherwise the block-tree recursion continues until both clusters are terminal. Every non-diagonal block, including inadmissible ones, is factorized; only diagonal blocks remain dense in FP64. We use ten total relative tolerances from $10^{-6}$ to $10^{-1}$ (see supplementary Section~\ref{sec:supp-setup} for details). Unless stated otherwise $B_k^{\mathrm{tot}}=\varepsilon\normmax{A_k}$. For each non-diagonal block, rank selection starts from a full dense SVD and requires its tail to satisfy both a relative Frobenius target of $0.1$ and an absolute allowance of $0.25B_k^{\mathrm{tot}}$. The measured max-norm residual $\eta_k$ is then subtracted from the budget, with 2\% reserved for arithmetic. Ranks may reach the smaller block dimension; low rank is an observed property, not an assumption.

The available integer bitwidths are $\{4,8,12,16,24,32\}$, with FP64 scales, group products, and accumulation.  Failed candidates are reported as infeasible, without dense non-diagonal substitution.  Independent arithmetic tests use both FP32 and FP64. Emulation forms integer codes and independently scaled decoded factors. The reported storage is calculated from the assigned integer bitwidths, FP64 scales, and bitwidth labels. 

For the implemented arithmetic bound, factor-representation rounding is included in addition to Theorem~\ref{thm:local-gemm}.  With $W_k=\sum_j\sigma_{k,j}\rho(b_{k,j})^2$, one rounding per factor entry gives
\begin{equation}
 R_k^{\mathrm{impl}}=(2u_{\mathrm a}+u_{\mathrm a}^2)W_k
                 +(1+u_{\mathrm a})^2R_k^{\mathrm{cert}}.
 \label{eq:exp-representation}
\end{equation}
The native decoder uses $\gamma_2(u_{\mathrm a})$ to account for scale multiplication and conversion (supplementary Section~\ref{sec:supp-arithmetic}).

\subsection{Local error bounds and allocation quality}
\label{sec:exp-local}

We compare Algorithm~\ref{alg:certified-greedy} with uniform bitwidths, the low to high greedy baseline without a future-error reserve, and a dynamic-programming (DP) optimum for the same discrete payload objective under the singular value error bound.  Across the 36 matrix tolerance pairs, the proposed allocator and DP have the same 24 feasible cases as uniform 16-bit allocation.  The remaining 12 cases fail the highest-bit feasibility test: even assigning 16 bits to every retained component cannot satisfy $E_{\mathrm{sv}}\le B^{\mathrm q}$.  The baseline without a reserve returns no feasible assignment in this sweep, including the 24 cases for which a feasible assignment exists.  Uniform 4- and 8-bit assignments are also infeasible throughout, whereas uniform 12-bit succeeds in four cases.  Thus the future-error reserve ensures feasibility under the prescribed error bound. The feasibility threshold remains determined by the highest available bitwidth.

For example, with eight unit singular values, $\Bset=\{4,8,16\}$, and $B^{\mathrm q}=6c(4)+c(16)/2$, the baseline assigns six components 4 bits, leaving only $c(16)/2$ for two components that require at least $2c(16)$.  The proposed rule assigns five components 4 bits and three 8 bits and remains feasible; supplementary Figure~\ref{fig:supp-reserve-boundary} shows the corresponding slack trace.

For all 76 feasible local results, including the baselines, the measured reconstruction errors remain below the corresponding singular value-based bounds.  For the 24 proposed allocations, the largest ratio of measured error to $\sum_j\sigma_jc(b_j)$ is $0.0162$ (rounded upward).  Evaluating the geometry-aware bound on the \emph{same} assignments reduces the bound by factors between 25.8 and 54.2, without a violation. The difference between the two bounds comes from retaining $\alpha_j\beta_j$ in the geometry-aware bound; the assignments still use the ordered weights $\sigma_j$. Figure~\ref{fig:exp-local-bounds} displays the normalized errors and both bounds for $n=256$.  The gap between the two bounds identifies a source of potential improvement in allocation, while the remaining gap to measured error cautions against treating either bound as an error prediction.

\begin{figure}[tbp]
\centering
\includegraphics[width=0.81\linewidth]{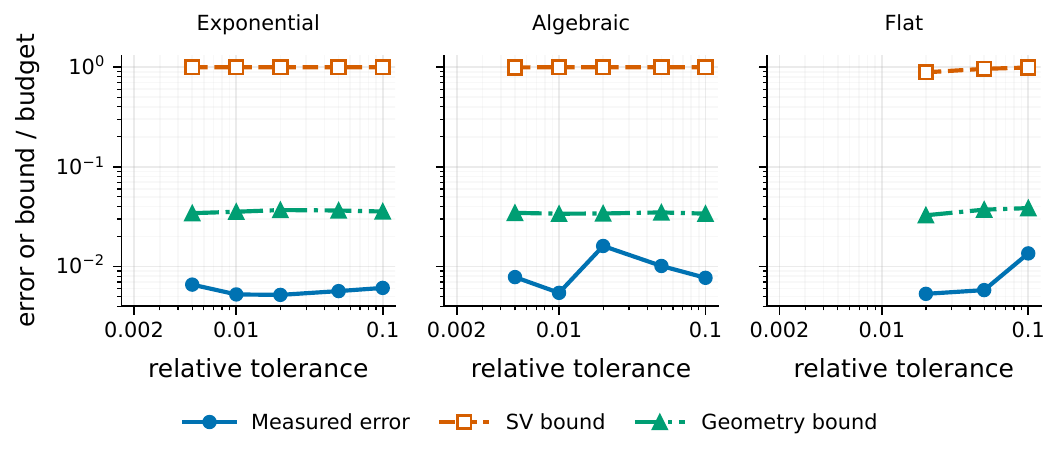}
\caption{Local reconstruction errors and error bounds divided by $B^{\mathrm q}=\varepsilon\|T\|_{\max}$, for $n=256$ and rank 64. Both bounds are evaluated on the same greedy assignments, which use singular value weights.  Missing points are referred to as those do not satisfy $E_{\mathrm{sv}}\le B^{\mathrm q}$.}
\label{fig:exp-local-bounds}
\end{figure}

Adaptivity is most useful for decaying spectra.  At $\varepsilon=0.1$, the exponential cases require only 52.3--54.3\% of the uniform 16-bit payload, whereas the flat cases require 94.9--96.1\%.  Figure~\ref{fig:exp-allocator-storage} shows the corresponding tolerance dependence, normalized by the bits used by uniform 16-bit factors rather than dense matrix storage, thereby removing the common low-rank compression factor.  Dominant components retain high precision while a low-precision tail absorbs the remaining error budget.  Figure~\ref{fig:paper-precision-map} shows the assignments: for the exponential case at $n=256$, $\varepsilon=0.1$, 20, 2, 6, and 36 components receive 16, 12, 8, and 4 bits, respectively.  Maps need not be nested across tolerances; monotonicity in singular value index within an assignment is a different property.  Supplementary Section~\ref{sec:supp-local} gives the full local validation and the $n=512$ storage curves.

\begin{figure}[tbp]
\centering
\includegraphics[width=0.81\linewidth]{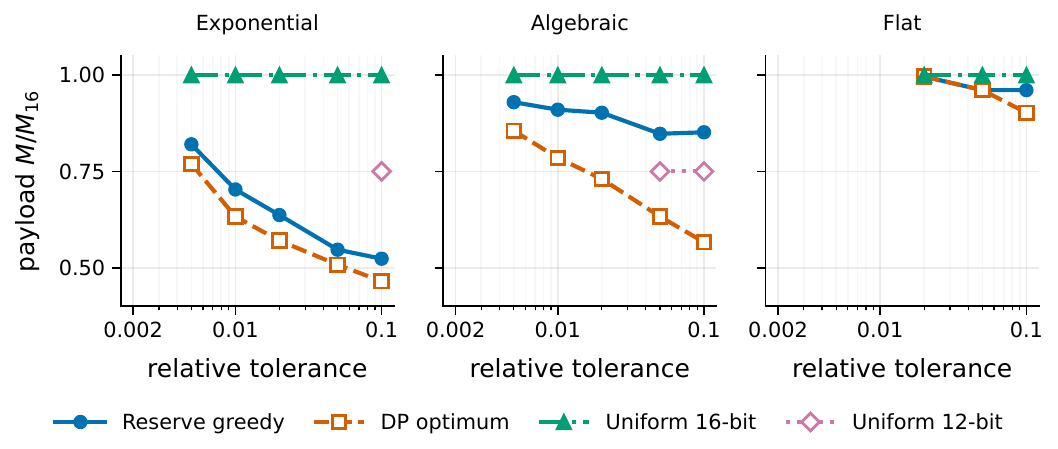}
\caption{Local integer payload relative to uniform 16-bit factors for $n=256$ and $r=64$.  Only assignments satisfying $E_{\mathrm{sv}}\le B^{\mathrm q}$ are shown; infeasible points are omitted.  Uniform 4- and 8-bit allocations and the baseline without a reserve have no feasible points.  Scale and precision-map metadata are excluded.  The algebraic spectrum shows that greedy need not dominate a feasible uniform precision.}
\label{fig:exp-allocator-storage}
\end{figure}

\begin{figure}[tbp]
\centering
\includegraphics[width=0.85\linewidth]{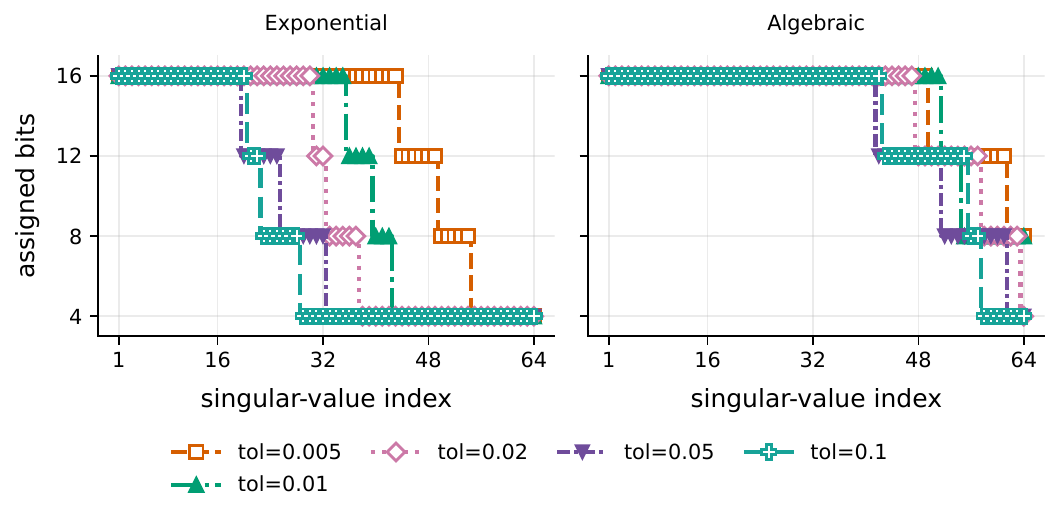}
\caption{Representative greedy bitwidth assignments for $n=256$, $r=64$. Indices follow decreasing singular value. The exponential and algebraic cases show different transitions; the latter also exhibits the payload behavior in supplementary Table~\ref{tab:supp-nonmonotone}.}
\label{fig:paper-precision-map}
\end{figure}

We next compare the factor bits used by greedy with those of the DP-optimal assignment.  Let $M_{\mathrm G}$ and $M_*$ be the bits used by the factors under the proposed greedy and DP-optimal assignments, respectively. Over the 24 matched pairs, $M_{\mathrm G}/M_*$ has median 1.091 and maximum 1.503.  The maximum occurs for the algebraic spectrum with $n=256$ at $\varepsilon=0.1$. Theorem~\ref{thm:greedy-feasible} guarantees feasibility; the ratios above show how the greedy payload compares with the DP optimum under the same error bound.

The uniform baseline also matters.  At $\varepsilon=0.1$, uniform 12-bit allocation is feasible for both algebraic-spectrum cases, but greedy uses 13.54\% and 14.58\% more payload for $n=256$ and $512$, respectively. In these cases, uniform 12-bit factors use less storage than the greedy assignment, despite its savings relative to uniform 16-bit factors. For the algebraic cases, greedy payload also increases when the tolerance changes from $0.05$ to $0.1$, and the  payloads is given  in supplementary Table~\ref{tab:supp-nonmonotone}.  The feasible sets expand with the budget, but the recomputed greedy choices need not improve the objective.  This is a property of the heuristic, not of the exact discrete optimum, whose payload cannot increase when its budget is relaxed.

\subsection{\texorpdfstring{Sensitivity to available bitwidths}{Sensitivity to available bitwidths}}
\label{sec:exp-menu}

A separate FP64 emulation study varies the set of available bitwidths while keeping each matrix and its budget fixed.  It uses $200\times180$ matrices of rank 25, ten random seeds, five spectra, and nine tolerances, giving 450 matrix tolerance pairs.  The spectra comprise three exponential decays, an algebraic decay, and a flat spectrum; supplementary Section~\ref{sec:supp-menu} specifies the reproducible configuration. The nested sets of bitwidths are
\begin{align*}
 B_0&=\{4,8,12,16\}, & B_1&=B_0\cup\{6,10\},\\
 B_2&=B_1\cup\{24,32\}, & B_3&=B_2\cup\{2\}.
\end{align*}
Comparisons use bits for the integer factor entries, excluding metadata, with the same constraint $E_{\mathrm{sv}}\le B^{\mathrm q}$ for greedy, DP, and the best feasible uniform precision.

However, increasing the available bitwidths does not necessarily improve the heuristic. From $B_0$ to $B_1$, the highest-bit feasibility threshold is unchanged: both sets admit 180 pairs.  Nevertheless, greedy payload increases in 153 of 180 matched pairs, by at most 24.65\% (rounded upward), whereas DP payload never increases.  Adding 24- and 32-bit choices admits all 450 pairs, including 270 newly feasible pairs.  On the original 180 pairs, however, greedy uses more payload in 177 cases, with a maximum increase of 66.87\%.  Adding the 2-bit choice changes no feasibility decision and can either increase or decrease greedy payload.

\begin{figure}[tbp]
\centering
\includegraphics[width=0.81\linewidth]{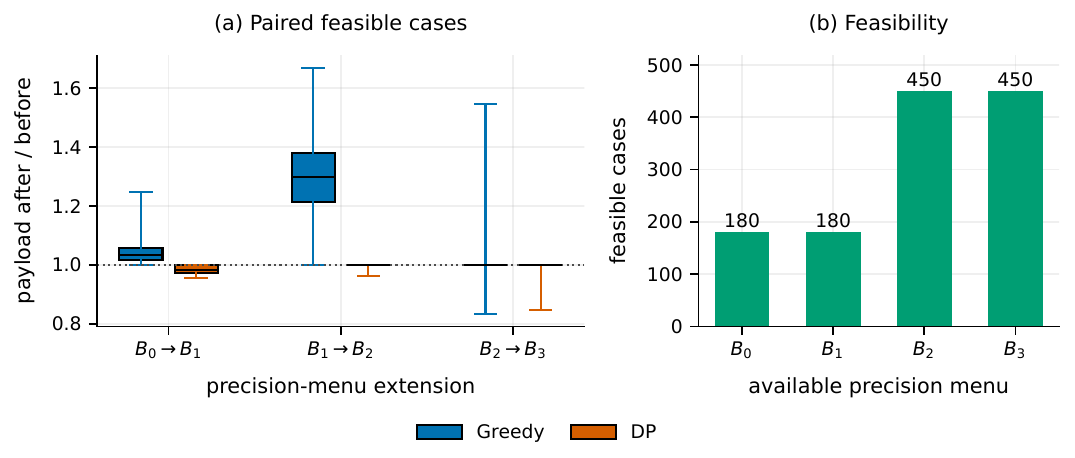}
\caption{Sensitivity to available bitwidths in the independent rank-25 study. Left: paired payload ratios over cases feasible before and after each extension (180, 180, and 450 pairs, respectively); boxes show quartiles and medians, and whiskers span the full range.  Newly feasible cases are excluded from these ratios.  Right: feasibility counts out of 450, identical for the proposed greedy method, DP, and the best feasible uniform precision. The dotted line on the left denotes unchanged payload.}
\label{fig:exp-menu}
\end{figure}

The distinction is between a larger feasible set and the decisions of a particular heuristic.  A lower highest-bit coefficient permits more aggressive early commitments to the low-precision tail; repairing the remaining prefix can then be expensive.  DP cannot deteriorate under these nested sets, since every previous assignment remains admissible. For the same reason, retaining a previously feasible assignment as an incumbent would guard against payload regression, without implying DP optimality.  Supplementary Sections~\ref{sec:supp-menu} and \ref{sec:supp-reserve} give spectrum-resolved statistics and a reserve trace, and distinguish the exact feasibility boundary from its floating-point realization.  No runtime conclusion is drawn from this study.

\subsection{Hierarchical reconstruction errors and storage reduction}
\label{sec:exp-offdiagonal}
\label{sec:exp-hierarchical}

All $44\times10\times3=1{,}320$ combinations of matrix, tolerance, and block budget are feasible when only diagonal blocks are stored densely, with no off-diagonal dense fallback.  None violates $e=\normmax{A-\widetilde A}\le E=\max_{k\in\LR}(\eta_k+Q_k^{\mathrm{cert}}+R_k^{\mathrm{impl}})$ with validation allowance $10^{-12}$.  All measured errors are nonzero; for block-relative budgets, $E/e$ ranges from 4.39 to 36.75. Thus the error bound is conservative but remains informative over five orders of magnitude in prescribed tolerance.  The quantization term is at least 86.0\% of the dominant block's error bound.  

Figure~\ref{fig:exp-global} shows the circle BEM matrix at $N=2048$. The measured error stays below the error bound throughout the sweep. Here, as in the other BEM tests, the reference is the assembled collocation matrix.

\begin{figure}[tbp]
\centering
\includegraphics[width=\linewidth]{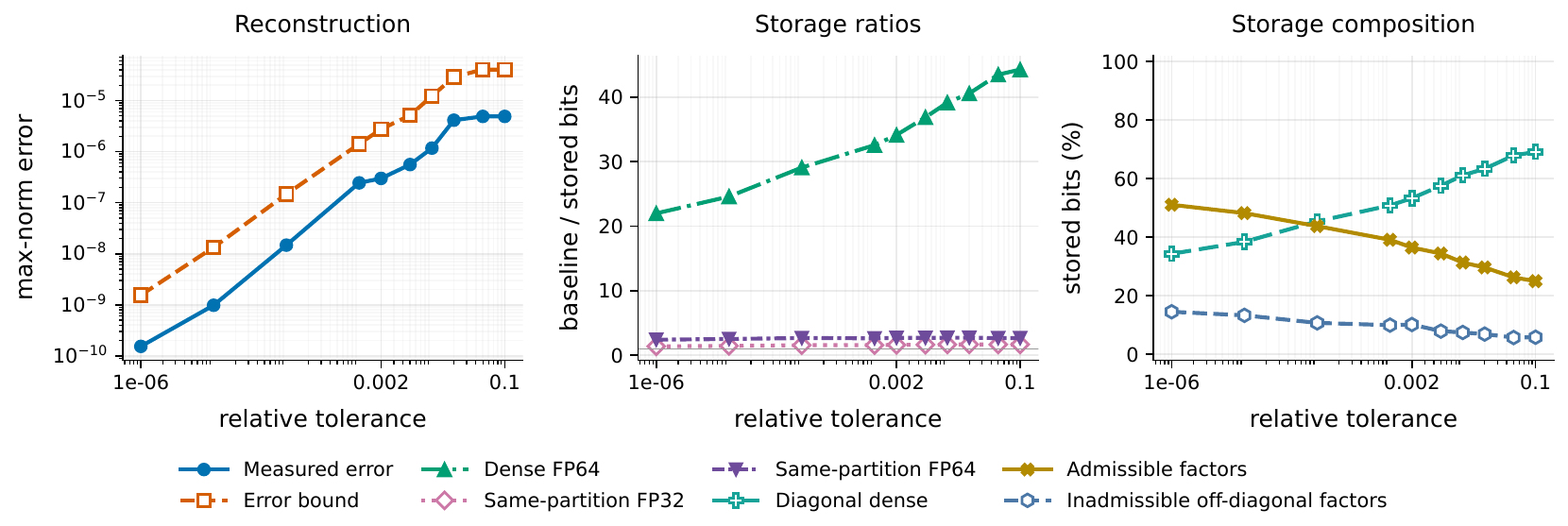}
\caption{Circle BEM, $N=2048$, with block-relative budgets and only diagonal blocks stored densely.  Left: absolute max-norm error and error bound. Middle: storage ratios against dense FP64 and same-partition floating factors.  Right: shares of total stored bits, including scales and maps, for the three block classes.  The horizontal tolerance is relative to each block norm.  All points represent feasible quantized constructions.}
\label{fig:exp-global}
\end{figure}

Let $M_{\mathrm q}$ include bits for the quantized factors, FP64 diagonal blocks, column scales, and precision maps, but exclude tree topology, allocation overhead, and execution workspace.  We distinguish total compression from the incremental storage benefit of quantization:
\begin{equation*}
 C_\varepsilon=\frac{64N^2}{M_{\mathrm q}},\qquad
 G_{32}=\frac{D+32S}{M_{\mathrm q}},\qquad
 G_{64}=\frac{D+64S}{M_{\mathrm q}},
\end{equation*}
where $D$ is the diagonal dense cost and $S=\sum_{k\in\LR}(m_k+n_k)r_k$. The floating-factor controls retain the identical partition and ranks, absorb singular values, and require no quantization scales or precision maps. The ratios therefore quantify storage differences for the fixed partition and ranks.

\begin{table}[tbp]
\centering
\caption{Diagonal-only dense storage at $\varepsilon=0.01$ with block-relative budgets.  Each geometry uses its largest tested matrix; point clouds are uniform.  $C_{0.01}$ compares against dense FP64; $G_{32},G_{64}$ compare against same-partition, same-rank floating factors with unchanged FP64 diagonal blocks.  The quantized denominator includes scales and maps.}
\label{tab:exp-hierarchical}
\footnotesize
\begin{tabular}{lrrrrr}
\toprule
Matrix & $N$ & $C_{0.01}$ & $G_{32}$ & $G_{64}$ & $E/e$ \\
\midrule
2D BEM, circle & 2048 & 39.141 & 1.666 & 2.721 & 10.48 \\
2D BEM, ellipse & 2048 & 37.107 & 1.680 & 2.781 & 6.74 \\
3D BEM, cube & 1536 & 12.524 & 2.542 & 4.889 & 9.29 \\
3D BEM, cuboid & 1536 & 12.908 & 2.454 & 4.706 & 13.18 \\
2D uniform, $1/r$ & 2048 & 19.133 & 2.437 & 4.575 & 9.92 \\
2D uniform, $\log r$ & 2048 & 19.873 & 2.379 & 4.448 & 8.08 \\
3D uniform, $1/r$ & 2048 & 9.828 & 2.757 & 5.361 & 15.07 \\
3D uniform, $\log r$ & 2048 & 8.890 & 2.714 & 5.288 & 11.97 \\
\bottomrule
\end{tabular}
\end{table}

Table~\ref{tab:exp-hierarchical} gives $C_{0.01}=8.890$--$39.141$, whereas the same-partition gains are $G_{32}=1.666$--$2.757$ and $G_{64}=2.721$--$5.361$.  Circle BEM has the largest dense-baseline compression but the smallest same-partition gain among these cases. Conversely, the three dimensional $1/r$ point kernel has substantially less total compression but the largest incremental quantization gain. Geometric compressibility and precision reduction are therefore distinct effects; the dense-baseline ratio alone does not measure the contribution of adaptive bit allocation.

The storage decomposition identifies different remaining limitations. At $\varepsilon=0.01$, diagonal blocks occupy 61.16\% of the circle BEM representation, whereas non-diagonal inadmissible factors occupy 64.70\% and 62.56\% for the three dimensional $1/r$ and $\log r$ kernels. Consequently, reducing factor bitwidths alone has limited leverage in the former, while the latter retain a substantial near-field factor cost. Scale/map metadata is 3.69--17.92\% of storage across the table and cannot be neglected when many short factors are quantized.

The three dimensional $\log r$ kernel remains less compressible than circle BEM.  At $\varepsilon=10^{-6}$, the three dimensional BEM and point-kernel cases contain full-rank non-diagonal inadmissible blocks; at $\varepsilon=0.01$, their median rank fractions $r_k/\min(m_k,n_k)$ are 0.250--0.281.  Factorization therefore does not imply a per-block storage advantage.  Rank and precision both adapt with tolerance, so the storage changes cannot be attributed to bitwidths alone. Supplementary Figure~\ref{fig:supp-hmatrix-storage} and Table~\ref{tab:supp-nearfield-ranks} give the corresponding curves and rank diagnostics.

We also compare three quantizations of the same balanced factors.  Per-tensor quantization uses one scale for each whole factor and one bitwidth per block.  Fixed-bit per-channel quantization uses a scale for each factor column but a common width within a block. The proposed method uses a separate scale and adaptive width for each column.  Each fixed-bit control chooses the lowest feasible width separately for each non-diagonal block.  All three methods share the partition, retained ranks, and block-relative budgets, and their total storage includes the common FP64 diagonal blocks, scales, and precision maps.

Across 60 matched matrix tolerance pairs, the adaptive representation uses less storage than fixed-bit per-channel quantization in every pair (median fixed-bit/adaptive ratio 1.252), but less than per-tensor quantization in only 45 pairs.  For the three $N=2048$ cases in Figure~\ref{fig:exp-granularity}, the per-channel/adaptive ratios narrow from 1.42--1.68 at $\varepsilon=10^{-6}$ to 1.10--1.37 at $\varepsilon=0.1$; the per-tensor/adaptive ratios narrow from 1.13--1.66 to 0.93--0.97.  Thus the endpoint storage differences diminish as tolerance is relaxed, although discrete bit and rank decisions make individual steps nonmonotone. Each control satisfies its own sufficient error bound while the measured reconstruction errors may differ. Complete comparison is given in supplementary Section~\ref{sec:supp-granularity}.

\begin{figure}[tbp]
\centering
\includegraphics[width=0.81\linewidth]{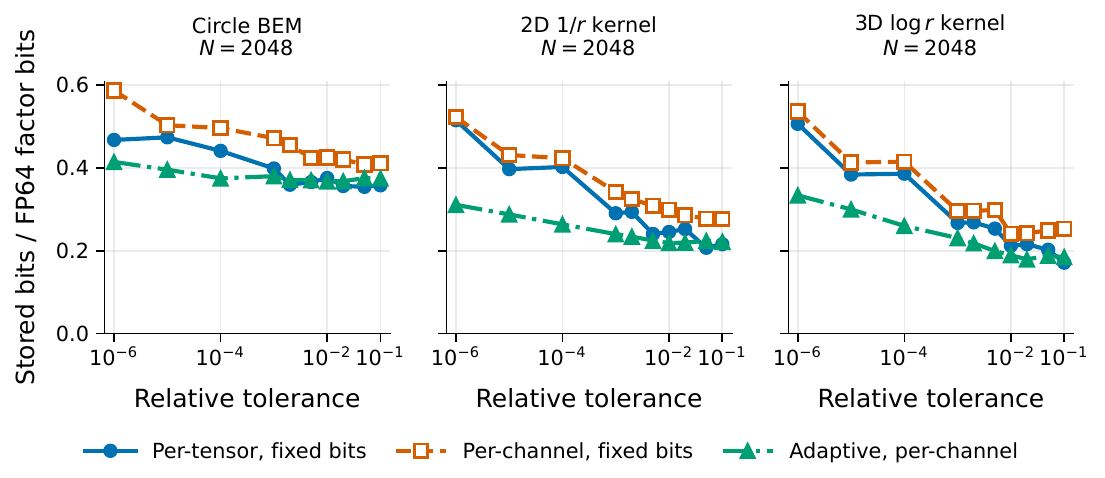}
\caption{Total stored bits relative to the same-partition FP64 factor representation for circle BEM, uniform two-dimensional $1/r$, and uniform three dimensional $\log r$ matrices at $N=2048$.  All methods use identical ranks at each block-relative tolerance.  The numerator includes FP64 diagonal blocks, factor scales, and precision maps.  Fixed-bit controls choose the lowest feasible width per block; their realized errors need not match the adaptive method's.}
\label{fig:exp-granularity}
\end{figure}

For all 132 combinations of matrix and block budget, total compression is non-decreasing over the tested tolerances; no general monotonicity of the greedy allocator follows. Comparisons of block budgets for the same matrix also require care: supplementary Table~\ref{tab:supp-policies} shows that hybrid budgets can reduce storage without changing the global error or its bound, whereas global-relative budgets permit larger errors.  This is consistent with the maximum-over-blocks error bound, not an equal-accuracy ranking of the budget choices. The studies that retain inadmissible blocks densely and those that store only diagonal blocks densely also differ in rank selection, available bitwidths, and arithmetic.  Therefore, their comparison does not isolate the effect of near-field factorization.

We also evaluated $Y=AX$ using the representation that stores only diagonal blocks densely. Factorized blocks are applied without forming their dense products. In all 3,600 runs, the observed arithmetic errors were below their respective bounds, as were the application errors bounded by propagating the measured block residuals (supplementary Section~\ref{sec:supp-application}).

\section{Conclusion}\label{sec:conclude}

We developed an adaptive method for quantizing hierarchical matrices based on an allocation strategy that leverages information on singular values. It proceeds by allocating bitwidths subject to a reconstruction error bound for the non-diagonal blocks of hierarchical matrices, with only diagonal blocks retained densely and unquantized.  With balanced column-wise quantization the strategy yields an additive max-norm error bound, and reserving the minimum future error makes the greedy allocation feasible whenever the highest available precision is feasible.  Combining truncation, quantization, and arithmetic bounds over a disjoint partition of matrix entries establishes a global reconstruction error bound. Numerical experiments show the storage reductions obtained by quantization and quantify the effectiveness  of the error bounds.  Future work includes error analysis for complete hierarchical operations and hardware-aware execution. 

\bibliographystyle{siamplain}
\bibliography{references}

@article{hackbusch2000sparse,
  title={A sparse $\mathcal{H}$-matrix arithmetic. Part II: Application to multi-dimensional problems},
  author={Hackbusch, Wolfgang and Khoromskij, Boris N},
  journal={Computing},
  volume={64},
  number={1},
  pages={21--47},
  year={2000},
  publisher={Springer}
}

@article{weak_hackbusch2004,
  title={Hierarchical matrices based on a weak admissibility criterion},
  author={Hackbusch, Wolfgang and Khoromskij, Boris N and Kriemann, Ronald},
  journal={Computing},
  volume={73},
  number={3},
  pages={207--243},
  year={2004},
  publisher={Springer}
}

@article{borm2003hierarchical,
  title={Hierarchical matrices},
  author={B{\"o}rm, Steffen and Grasedyck, Lars and Hackbusch, Wolfgang},
  journal={Lecture notes},
  volume={21},
  pages={2003},
  year={2003},
  publisher={Citeseer}
}

@article{borm2009construction,
  title={Construction of data-sparse $\mathcal{H}^2$-matrices by hierarchical compression},
  author={B{\"o}rm, Steffen},
  journal={SIAM Journal on Scientific Computing},
  volume={31},
  number={3},
  pages={1820--1839},
  year={2009},
  publisher={SIAM}
}

@article{ambikasaran2019hodlrlib,
  title={{HODLRlib}: {A} library for hierarchical matrices},
  author={Ambikasaran, Sivaram and Singh, Karan Raj and Sankaran, Shyam Sundar},
  journal={Journal of Open Source Software},
  volume={4},
  number={34},
  pages={1167},
  year={2019}
}

@article{gillman2012direct,
  title={A direct solver with $\mathcal{O} ({N})$ complexity for integral equations on one-dimensional domains},
  author={Gillman, Adrianna and Young, Patrick M and Martinsson, Per-Gunnar},
  journal={Frontiers of Mathematics in China},
  volume={7},
  number={2},
  pages={217--247},
  year={2012},
  publisher={Springer}
}

@article{chandrasekaran2007fast,
  title={A fast solver for {HSS} representations via sparse matrices},
  author={Chandrasekaran, Shiv and Dewilde, Patrick and Gu, Ming and Lyons, William and Pals, Timothy},
  journal={SIAM Journal on Matrix Analysis and Applications},
  volume={29},
  number={1},
  pages={67--81},
  year={2007},
  publisher={SIAM}
}

@inproceedings{10.5555/3666122.3666573,
author = {Wortsman, Mitchell and Dettmers, Tim and Zettlemoyer, Luke and Morcos, Ari and Farhadi, Ali and Schmidt, Ludwig},
title = {Stable and low-precision training for large-scale vision-language models},
year = {2023},
booktitle = {Advances in Neural Information Processing Systems},
volume = {36},
pages = {10271--10298},
doi = {10.52202/075280-0451}
}

@INPROCEEDINGS{10028742,
  author={Kulkarni, Uday and Hosamani, Abhishek S and Masur, Abhishek S and Hegde, Shashank and Vernekar, Ganesh R and Siri Chandana, K},
  booktitle={2022 International Conference on Automation, Computing and Renewable Systems (ICACRS)}, 
  title={A Survey on Quantization Methods for Optimization of Deep Neural Networks}, 
  year={2022},
  pages={827--834},
  publisher={IEEE},
  doi={10.1109/ICACRS55517.2022.10028742}
}

@article{10.1145/3623402,
author = {Rokh, Babak and Azarpeyvand, Ali and Khanteymoori, Alireza},
title = {A Comprehensive Survey on Model Quantization for Deep Neural Networks in Image Classification},
year = {2023},
publisher = {ACM},
volume = {14},
number = {6},
articleno = {97},
pages = {97:1--97:50},
journal = {ACM Transactions on Intelligent Systems and Technology},
numpages = {50},
doi = {10.1145/3623402}
}

@inproceedings{10.1145/3123939.3123982,
author = {Albericio, Jorge and Delm\'{a}s, Alberto and Judd, Patrick and Sharify, Sayeh and O'Leary, Gerard and Genov, Roman and Moshovos, Andreas},
title = {Bit-pragmatic deep neural network computing},
year = {2017},
publisher = {ACM},
booktitle = {Proceedings of the 50th Annual IEEE/ACM International Symposium on Microarchitecture},
pages = {382--394},
numpages = {13},
series = {MICRO-50 '17},
doi = {10.1145/3123939.3123982}
}

@INPROCEEDINGS{7551407,
  author={Chen, Yu-Hsin and Emer, Joel and Sze, Vivienne},
  booktitle={2016 ACM/IEEE 43rd Annual International Symposium on Computer Architecture (ISCA)}, 
  title={Eyeriss: A Spatial Architecture for Energy-Efficient Dataflow for Convolutional Neural Networks}, 
  year={2016},
  pages={367--379},
  publisher={IEEE},
  doi={10.1109/ISCA.2016.40}
}

@article{10.1093/imanum/drac037,
    author = {Amestoy, Patrick and Boiteau, Olivier and Buttari, Alfredo and Gerest, Matthieu and Jézéquel, Fabienne and L’Excellent, Jean-Yves and Mary, Theo},
    title = {Mixed precision low-rank approximations and their application to block low-rank LU factorization},
    journal = {IMA Journal of Numerical Analysis},
    volume = {43},
    number = {4},
    pages = {2198--2227},
    year = {2023},
    month = {07},
    doi = {10.1093/imanum/drac037}
}

@book{hackbusch2015hierarchical,
  author = {Hackbusch, Wolfgang},
  title = {Hierarchical Matrices: Algorithms and Analysis},
  series = {Springer Series in Computational Mathematics},
  volume = {49},
  publisher = {Springer},
  address = {Berlin, Heidelberg},
  year = {2015},
  doi = {10.1007/978-3-662-47324-5}
}

@article{hackbusch1999sparse,
  author = {Hackbusch, Wolfgang},
  title = {A Sparse Matrix Arithmetic Based on {H}-Matrices. Part {I}: Introduction to {H}-Matrices},
  journal = {Computing},
  volume = {62},
  number = {2},
  pages = {89--108},
  year = {1999},
  doi = {10.1007/s006070050015}
}

@book{bebendorf2008hierarchical,
  author = {Bebendorf, Mario},
  title = {Hierarchical Matrices: A Means to Efficiently Solve Elliptic Boundary Value Problems},
  series = {Lecture Notes in Computational Science and Engineering},
  volume = {63},
  publisher = {Springer},
  address = {Berlin, Heidelberg},
  year = {2008},
  doi = {10.1007/978-3-540-77147-0}
}

@INPROCEEDINGS{8310262,
  author={Lee, Jinmook and Kim, Changhyeon and Kang, Sanghoon and Shin, Dongjoo and Kim, Sangyeob and Yoo, Hoi-Jun},
  booktitle={2018 IEEE International Solid-State Circuits Conference (ISSCC)}, 
  title={{UNPU}: A 50.6{TOPS/W} unified deep neural network accelerator with 1b-to-16b fully-variable weight bit-precision}, 
  year={2018},
  publisher={IEEE},
  pages={218--220},
  doi={10.1109/ISSCC.2018.8310262}
}

@inproceedings{ooi2020effect,
  author    = {Rise Ooi and Takeshi Iwashita and Takeshi Fukaya and Akihiro Ida and Rio Yokota},
  title     = {Effect of Mixed Precision Computing on {H}-Matrix Vector Multiplication in {BEM} Analysis},
  booktitle = {Proceedings of the International Conference on High Performance Computing in Asia-Pacific Region},
  pages     = {92--101},
  year      = {2020},
  publisher = {ACM},
  doi       = {10.1145/3368474.3368479}
}

@article{carson2025mixedhodlr,
  author = {Carson, Erin and Chen, Xinye and Liu, Xiaobo},
  title = {Mixed Precision {HODLR} Matrices},
  journal = {SIAM Journal on Scientific Computing},
  volume = {47},
  number = {3},
  pages = {A1408--A1435},
  year = {2025},
  doi = {10.1137/24M1683925}
}

@misc{khan2026hybrid,
  author        = {Ritesh Khan and Erin Carson},
  title         = {Hybrid Hierarchical Matrices with Adaptive Mixed Precision Storage},
  year          = {2026},
  eprint        = {2604.09147},
  archivePrefix = {arXiv},
  primaryClass  = {math.NA}
}

@article{kriemann2025hierarchical,
  author  = {Ronald Kriemann},
  title   = {Hierarchical Low-Rank Arithmetic with Floating Point Compression},
  journal = {SIAM Journal on Scientific Computing},
  volume  = {47},
  number  = {4},
  pages   = {B763--B784},
  year    = {2025},
  doi     = {10.1137/24M1649009}
}

@inproceedings{li2023ivit,
  title     = {{I-ViT}: Integer-Only Quantization for Efficient Vision Transformer Inference},
  author    = {Li, Zhikai and Gu, Qingyi},
  booktitle = {Proceedings of the IEEE/CVF International Conference on Computer Vision (ICCV)},
  pages     = {17065--17075},
  month     = {October},
  year      = {2023},
  doi       = {10.1109/ICCV51070.2023.01565}
}

@inproceedings{kim2021ibert,
  title     = {{I-BERT}: Integer-Only BERT Quantization},
  author    = {Kim, Sehoon and Gholami, Amir and Yao, Zhewei and Mahoney, Michael W. and Keutzer, Kurt},
  booktitle = {Proceedings of the 38th International Conference on Machine Learning},
  pages     = {5506--5518},
  year      = {2021},
  volume    = {139},
  series    = {Proceedings of Machine Learning Research},
  publisher = {PMLR}
}

@incollection{gholami2021survey,
  title     = {A Survey of Quantization Methods for Efficient Neural Network Inference},
  author    = {Gholami, Amir and Kim, Sehoon and Dong, Zhen and Yao, Zhewei and Mahoney, Michael W. and Keutzer, Kurt},
  booktitle = {Low-Power Computer Vision},
  pages     = {291--326},
  publisher = {Chapman and Hall/CRC},
  year      = {2022},
  doi       = {10.1201/9781003162810-13}
}

\clearpage
\appendix
\renewcommand{\appendixname}{Supplementary Material}
The supplementary materials provide reproducibility details and extended diagnostics for Section~\ref{sec:exps} of the main paper. The main paper presents the local and global error bounds, allocation comparisons, the effects of geometry on storage, and storage comparisons on the same block partitions. Further results for the fixed-bit quantization controls and native GPU numerical checks are collected here. The reconstruction and application checks use separate validation suites and distinct references.

\section{Test definitions and reproducibility}\label{sec:supp-setup}

The studies of local allocation, available bitwidths, and hierarchical storage and error use Python and NumPy on the CPU. Reconstruction uses FP64 arithmetic unless an FP32 arithmetic control is specified. Python 3.12.11 and NumPy 2.5.3 were used for the numerical validation and allocation studies. The native reconstruction and matrix application checks use C++17, the CUDA 13.2 runtime, and cuBLAS on a single NVIDIA A100 GPU. Code, configurations, and reproduction instructions are available at \url{https://github.com/inEXASCALE/qhmat}.

Random seeds are fixed by the experiment configurations; the study across ten seeds specifies them in Section~\ref{sec:supp-menu}. The single-block suite has $n\in\{256,512\}$, rank 64, and exponential, algebraic, and flat spectra as specified in the main paper.  The single-block comparison uses six tolerances $0.002,0.005,0.01,0.02,0.05,0.1$ and bitwidths $\{4,8,12,16\}$.  The hierarchical study in which only diagonal blocks remain dense uses ten tolerances $10^{-6},10^{-5},10^{-4},0.001,0.002,0.005,0.01,0.02,0.05,0.1$ and the available bitwidths $\{4,8,12,16,24,32\}$.

The hierarchical storage totals charge two FP64 scales per retained component and a direct precision map recording each assigned bitwidth, at a cost of $r_k\lceil\log_2|\Bset|\rceil$ bits for each factorized block $k$. By Theorem~\ref{thm:exchange}, a monotone assignment could instead be encoded by the bitwidth labels of its $K$ nonempty groups and the $K-1$ boundaries between them.
Tree topology, serialization headers, and runtime workspace are excluded. The bounds treat scales as exact or sufficiently accurate; more aggressive scale quantization would require additional error terms.

The 44 hierarchical matrices comprise six one-dimensional kernels, 24 point-kernel matrices, eight planar BEM matrices, and six surface BEM matrices.  For the one-dimensional cases, $x_i=(i-1)/(N-1)$, $N\in\{512,1024\}$, $r_{ij}=|x_i-x_j|$, and $h=1/(N-1)$.  The matrices are
\begin{align*}
 A^{\mathrm{inv}}_{ij}&=\frac{1}{r_{ij}+2h},&
 A^{\mathrm G}_{ij}&=\exp[-(r_{ij}/0.18)^2],\\
 A^{\mathrm M}_{ij}&=(1+z_{ij})e^{-z_{ij}},&
 z_{ij}&=\sqrt{3}r_{ij}/0.16.
\end{align*}
The last is the Mat\'ern-$3/2$ kernel.  The code identifier \texttt{laplace1d} denotes the regularized inverse-distance matrix above.
These kernels include their finite diagonal entries.

The point-kernel matrices use unregularized off-diagonal $1/r$ or $\log r$, with omitted self interactions $A_{ii}=0$.  Uniform clouds lie in $[-1/2,1/2]^d$; clustered clouds have centers $\pm0.3\boldsymbol{1}$ and coordinate standard deviation $0.07$.  The combinations of $d\in\{2,3\}$, uniform or clustered distributions, two kernels, and $N\in\{512,1024,2048\}$ give 24 cases.

The planar BEM boundaries are polygonal approximations of a circle of radius $0.5$ and an ellipse with semiaxes $(0.75,0.25)$, each with $N\in\{256,512,1024,2048\}$ panels. The surface boundaries are a unit cube and a cuboid with side lengths $(1.5,1,0.5)$. Subdividing each face into $s^2$ rectangles for $s\in\{4,8,16\}$ gives $N\in\{96,384,1536\}$. The constant-panel, midpoint-collocation matrices use the analytic single-layer integrals in the main paper, including finite singular self terms. No symmetrization is applied.

The hierarchical experiment builds the common row/column cluster tree by a stable median split along the coordinate of largest spread, stopping when a cluster contains at most 32 indices.  For BEM matrices, cluster boxes include full panel supports.  The admissibility test uses $\eta_{\mathrm{adm}}=1$ as specified in Section~\ref{sec:exp-setup} of the main paper; blocks that fail the test are subdivided until both clusters are terminal.  Only diagonal blocks remain dense.  The rank rule and FP64 arithmetic follow that section, with a 25\% truncation allowance and 2\% arithmetic reserve.  Failed non-diagonal candidates are reported as infeasible.

The hierarchical results include storage totals for floating-point factors on the same partitions. All error-bound checks use FP64 numerical references with the stated allowances. Section~\ref{sec:supp-application} specifies the application checks.

\section{Local error bounds and allocation}\label{sec:supp-local}

Theorem~\ref{thm:balanced-mixed} gives both the singular-value bound and the sharper bound involving the column maxima $\alpha_j\beta_j$. Figure~\ref{fig:supp-local-bounds} evaluates both on the same feasible assignments from the proposed greedy algorithm, dividing each error or bound by that case's budget $B^{\mathrm q}=\varepsilon\|T\|_{\max}$. The geometry bound is 25.8--54.2 times smaller than the singular-value bound across the 24 allocations. Both bounds are evaluated on assignments chosen using singular-value weights.

\begin{figure}[tbp]
\centering
\includegraphics[width=0.82\linewidth]{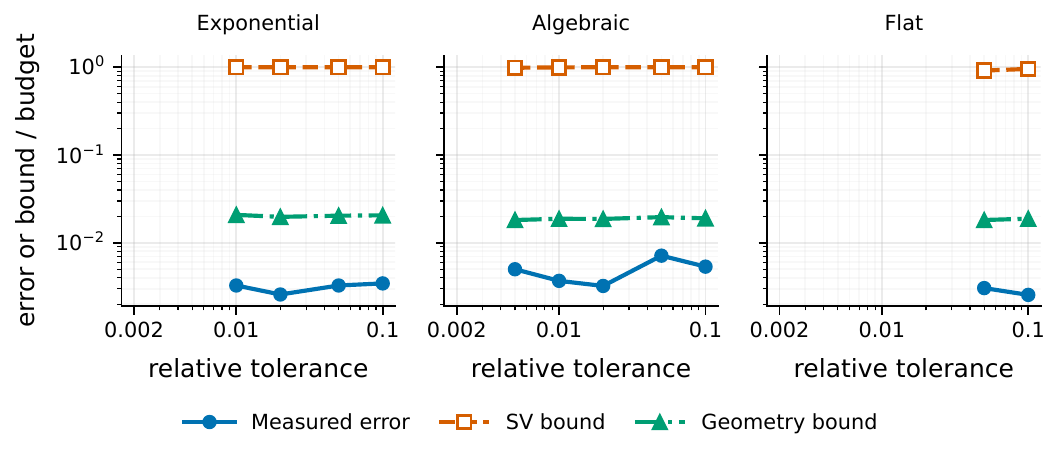}
\caption{Local errors and error bounds divided by the local budget, with $n=512$ and rank 64, supplementing the $n=256$ comparison in Figure~\ref{fig:exp-local-bounds} of the main paper. Both bounds use the same feasible greedy assignments. Omitted tolerances are infeasible under the available bitwidths.}
\label{fig:supp-local-bounds}
\end{figure}

Table~\ref{tab:exp-local} gives the casewise feasibility and DP comparisons. Greedy, DP, and the uniform 16-bit allocation share 24 feasible pairs; the other 12 fail even at the highest precision. The uniform 12-bit allocation succeeds in four pairs, whereas the uniform 4-bit and uniform 8-bit allocations and the low-to-high greedy baseline without a reserve have no feasible pairs. The latter rule fails on all 24 pairs admitting a feasible assignment.

\begin{table}[tbp]
\centering
\caption{Local allocation at rank 64.  Feasibility counts are out of six tolerances.  $M_{\mathrm G},M_{16},M_*$ denote greedy, uniform 16-bit, and DP-optimal integer payloads.  The fourth column uses $\varepsilon=0.1$; the last is the maximum over matched feasible tolerances.  Metadata is excluded from all three objectives.}
\label{tab:exp-local}
\small
\begin{tabular}{rlrrr}
\toprule
$n$ & Spectrum & Feasible & $M_{\mathrm G}/M_{16}$ & Max. $M_{\mathrm G}/M_*$ \\
\midrule
256 & Exponential & 5/6 & 0.523 & 1.126 \\
256 & Algebraic & 5/6 & 0.852 & 1.503 \\
256 & Flat & 3/6 & 0.961 & 1.065 \\
512 & Exponential & 4/6 & 0.543 & 1.112 \\
512 & Algebraic & 5/6 & 0.859 & 1.358 \\
512 & Flat & 2/6 & 0.949 & 1.000 \\
\bottomrule
\end{tabular}
\end{table}

Figure~\ref{fig:supp-allocator} completes the main paper's $n=256$ storage plot with $n=512$.  Decaying spectra allow substantial savings relative to the uniform 16-bit allocation, whereas flat spectra offer little component separation. The feasible uniform 12-bit algebraic case is cheaper than greedy at $\varepsilon=0.1$, even though greedy remains feasible.

\begin{figure}[tbp]
\centering
\includegraphics[width=0.82\linewidth]{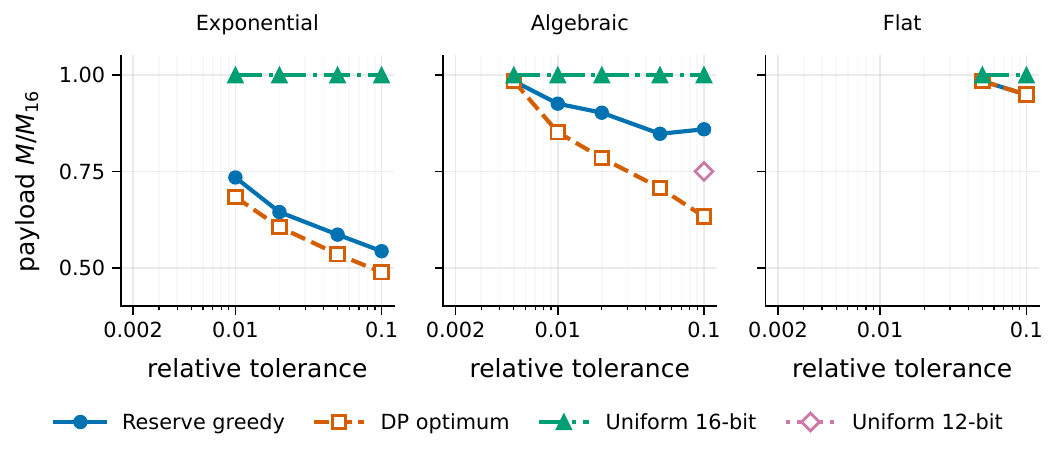}
\caption{Local payload comparison for $n=512$, $r=64$, normalized by uniform 16-bit factor payload, supplementing Figure~\ref{fig:exp-allocator-storage} of the main paper.  Missing assignments indicate that no admissible bit assignment exists; methods with no feasible points are omitted.}
\label{fig:supp-allocator}
\end{figure}

\begin{table}[tbp]
\centering
\caption{Nonmonotone greedy payload for the algebraic spectrum, in bits. Uniform12 is feasible at $\varepsilon=0.1$ for both sizes.}
\label{tab:supp-nonmonotone}
\begin{tabular}{rrrr}
\toprule
$n$ & Greedy, $0.05$ & Greedy, $0.1$ & Uniform12, $0.1$ \\
\midrule
256 & 444416 & 446464 & 393216 \\
512 & 888832 & 901120 & 786432 \\
\bottomrule
\end{tabular}
\end{table}

Table~\ref{tab:supp-nonmonotone} documents the payload increases when the budget is relaxed.  A larger feasible set cannot worsen the exact optimum, but a recomputed heuristic need not select a better point.  Figure~\ref{fig:paper-precision-map} of the main paper shows the corresponding bitwidth assignments.

\subsection{Available bitwidths and results across seeds}\label{sec:supp-menu}

For each seed in $\{1,7,42,123,202,303,404,505,606,707\}$, Gaussian QR factors have dimensions $200\times25$ and $180\times25$. The same factors are reused across spectra and sets of available bitwidths. Singular values are $\sigma_j=a^{j-1}$ for $a=0.55,0.75,0.90$, $\sigma_j=j^{-1.4}$, or $\sigma_j=1$. The tolerances are $10^{-5},3\cdot10^{-5},10^{-4},3\cdot10^{-4},10^{-3},3\cdot10^{-3}, 10^{-2},3\cdot10^{-2},10^{-1}$, with $B^{\mathrm q}=\varepsilon\|T\|_{\max}$. The sets $B_0,\ldots,B_3$ are defined in Section~\ref{sec:exp-menu} of the main paper. The statistics summarize these ten seeds.

All 3780 feasible combinations of methods, bitwidth sets, and test cases pass the numerical checks $e\le E_{\mathrm{geom}}\le E_{\mathrm{sv}}\le B^{\mathrm q}$, with allowance $512\epsilon_{\mathrm{mach}}\sum_j\sigma_j$ for FP64 evaluation. Here $\epsilon_{\mathrm{mach}}=2^{-52}$ is FP64 machine epsilon. Several combinations reuse the same matrices. The DP minimizes integer factor payload subject to the same error bound.

\begin{table}[tbp]
\centering\small
\caption{Greedy payload relative to DP in the rank-25 study of available bitwidths. Statistics use each set's feasible subset; Figure~\ref{fig:exp-menu} of the main paper uses paired subsets.}
\label{tab:supp-menu}
\begin{tabular}{lrrrr}
\toprule
Set & Feasible/450 & Median $M_{\mathrm G}/M_*$ & 90th percentile & Maximum\\
\midrule
$B_0$ & 180 & 1.076 & 1.197 & 1.554\\
$B_1$ & 180 & 1.135 & 1.301 & 1.596\\
$B_2$ & 450 & 1.273 & 1.658 & 2.587\\
$B_3$ & 450 & 1.314 & 1.745 & 2.587\\
\bottomrule
\end{tabular}
\end{table}

The $B_2\to B_3$ extension increases greedy payload in 80 of 450 paired cases, with maximum factor 1.546; its median factor is one. The three extensions isolate the effects of adding intermediate bitwidths, increasing the largest available bitwidth, and allowing 2-bit assignments. Figure~\ref{fig:supp-menu-spectra} shows the gaps between greedy and DP by spectrum.

\begin{figure}[tbp]
\centering
\includegraphics[width=0.85\linewidth]{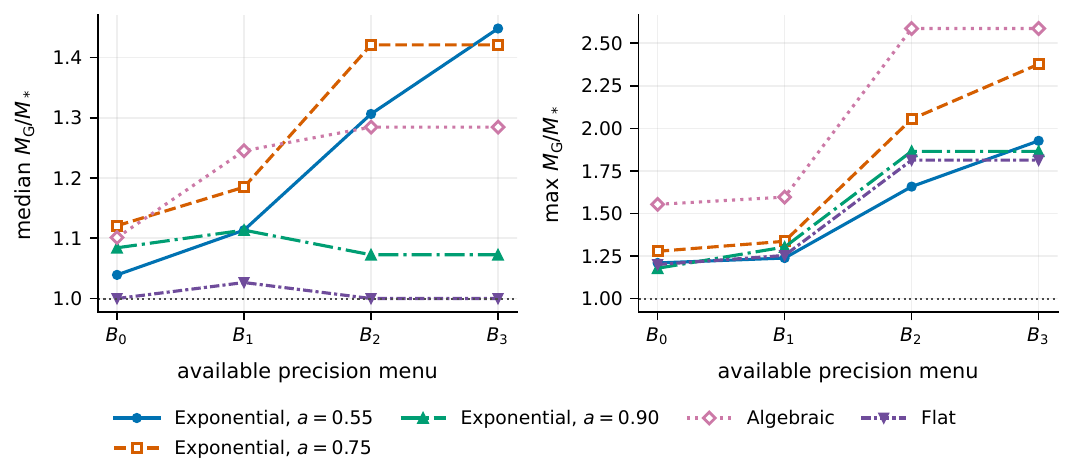}
\caption{Median and maximum greedy/DP payload ratios by spectrum and set of available bitwidths.  Each point aggregates feasible tolerances over ten seeds; the feasible subset can change when the largest available bitwidth changes. These descriptive curves complement the paired ratios in Figure~\ref{fig:exp-menu} of the main paper.}
\label{fig:supp-menu-spectra}
\end{figure}

\subsection{Reserve traces and the feasibility boundary}\label{sec:supp-reserve}

For a failure case, take eight unit singular values, $\mathcal B=\{4,8,16\}$, and $B^{\mathrm q}=6c(4)+c(16)/2$. The greedy baseline without a reserve assigns six components to 4 bits, leaving only $c(16)/2$ for two components whose minimum error contribution is $2c(16)$.  Its final error bound exceeds the budget by $3c(16)/2$. The proposed greedy algorithm instead assigns five components to 4 bits and three to 8 bits, leaving a nonnegative reserve.  The implementation optionally records the effective remaining budget $R-C$ and minimum future error after each accepted component.

\begin{figure}[tbp]
\centering
\includegraphics[width=0.85\linewidth]{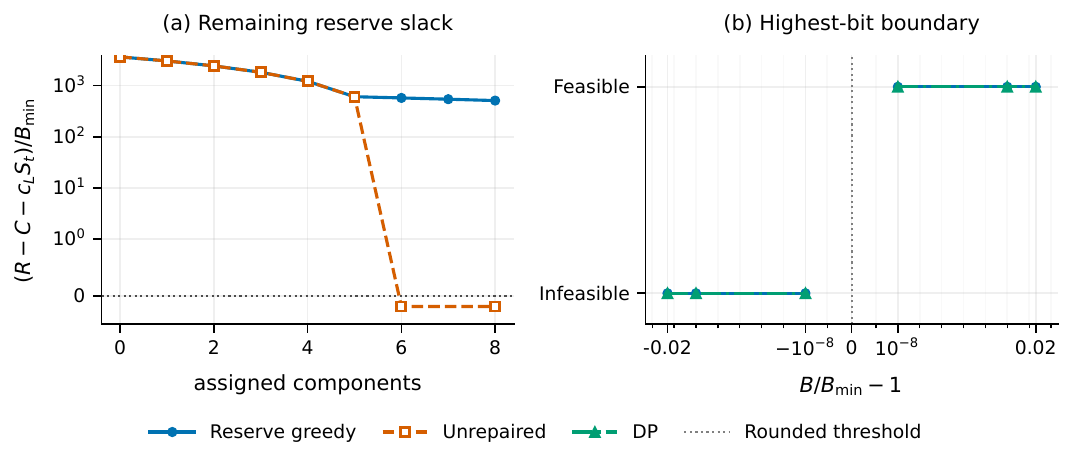}
\caption{Reserve mechanism and boundary diagnostics.  Left: remaining budget after accounting for future components, normalized by $B_{\min}=c(16)\sum_j\sigma_j$ in the eight-component failure case; the vertical scale is symmetric logarithmic with a linear region around zero.  Right: floating-point feasibility decisions for two unit singular values versus relative displacement from the exact threshold.  Both methods agree away from the central point, where rounding affects the decision; that point is omitted from the curves and marked by a vertical dotted line.}
\label{fig:supp-reserve-boundary}
\end{figure}

For the boundary test, use two singular values equal to $s\in\{2^{-40},1,2^{40}\}$. The exact threshold $B_{\min}=2s\bigl(q^{-1}+(4q^2)^{-1}\bigr)$, $q=32767$, is evaluated with rational arithmetic. Requested budget ratios are $0.98,0.999,1-10^{-8},1,1+10^{-8},1.001,1.02$; the algorithms receive their nearest FP64 budgets. Comparing those supplied budgets with the rational threshold, both methods agree with the exact predicate in all 36 evaluations outside the band $|B-B_{\min}|\le64\epsilon_{\mathrm{mach}}B_{\min}$. The six evaluations at the nominal central ratio are reported separately: the rounded budget lies just below the exact threshold, while both floating-point implementations classify it as feasible. Near this threshold, directed rounding or validated arithmetic would be required to classify the supplied budget rigorously.

\subsection{A reciprocal-scaling control}\label{sec:supp-scaling}

For column quantization with the scale $s_x=\|x\|_\infty/q_b$, write $\mathcal{Q}_b(x)=s_x\operatorname{round}(x/s_x)$, with zero mapped to zero. For $d>0$, $s_{dx}=d s_x$ and the normalized entries are unchanged. Thus, in exact arithmetic and for fixed bitwidth assignments,
\begin{equation}
 \mathcal{Q}_{b_j}(d_jp_j)
 \mathcal{Q}_{b_j}(d_j^{-1}q_j)^T
 =\mathcal{Q}_{b_j}(p_j)\mathcal{Q}_{b_j}(q_j)^T.
 \label{eq:supp-scaling-invariance}
\end{equation}
Summing over columns proves invariance of the quantized product under positive reciprocal diagonal rescaling.

The control uses the 50 $(U,\Sigma,V)$ triples from the ten seeds and five spectra in Section~\ref{sec:supp-menu}, with balanced factors $P=U\Sigma^{1/2}$ and $Q=V\Sigma^{1/2}$. The repeated bit pattern is $(4,8,12,16)$, truncated to rank 25. For $e\in\{-100,-50,0,50,100\}$, set $d_j=2^{-e+2e(j-1)/24}$ and define $D=\operatorname{diag}(d_1,\ldots,d_{25})$; reconstruct from $PD$ and $QD^{-1}$. FP64 scales and arithmetic avoid overflow and underflow in this sweep. The maximum change in the quantized product with separate column scales, divided by $\|T\|_{\max}$, is below $4.97\times10^{-16}$, whereas it can exceed one when a scale is shared by a bitwidth group. Figure~\ref{fig:supp-scaling} separates this change from the reconstruction error itself.

\begin{figure}[tbp]
\centering
\includegraphics[width=0.85\linewidth]{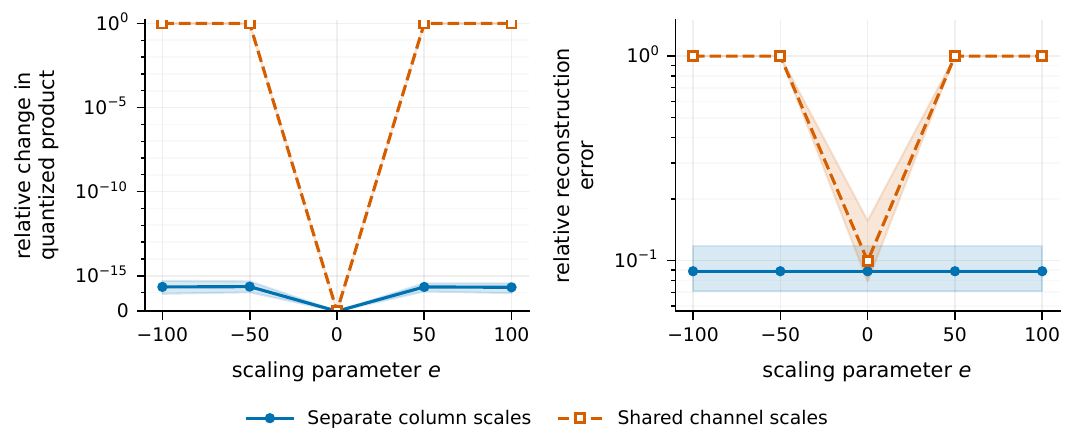}
\caption{Reciprocal positive diagonal rescaling with fixed bitwidth assignments.  Curves show medians and bands the full range over 50 matrices. Left: change from the corresponding unscaled quantized product, divided by $\|T\|_{\max}$.  Right: reconstruction error in the same normalization. The left axis is symmetric logarithmic near zero so exact zeros remain visible.  Separate column scales preserve the ideal product; scales shared within bitwidth groups provide a contrasting control.}
\label{fig:supp-scaling}
\end{figure}

\section{Finite-precision reconstruction}\label{sec:supp-arithmetic}

Theorem~\ref{thm:local-gemm} bounds rounding in the GEMM for each bitwidth group and in the subsequent summation. The implemented check also accounts for rounding when the factors are represented in the working precision, using equation~\eqref{eq:exp-representation} in the main paper. The native decoder uses the more conservative $\gamma_2(u_{\mathrm a})$ allowance for scale multiplication and conversion.  These bounds assume no overflow or underflow.

\begin{figure}[tbp]
\centering
\includegraphics[width=0.7\linewidth]{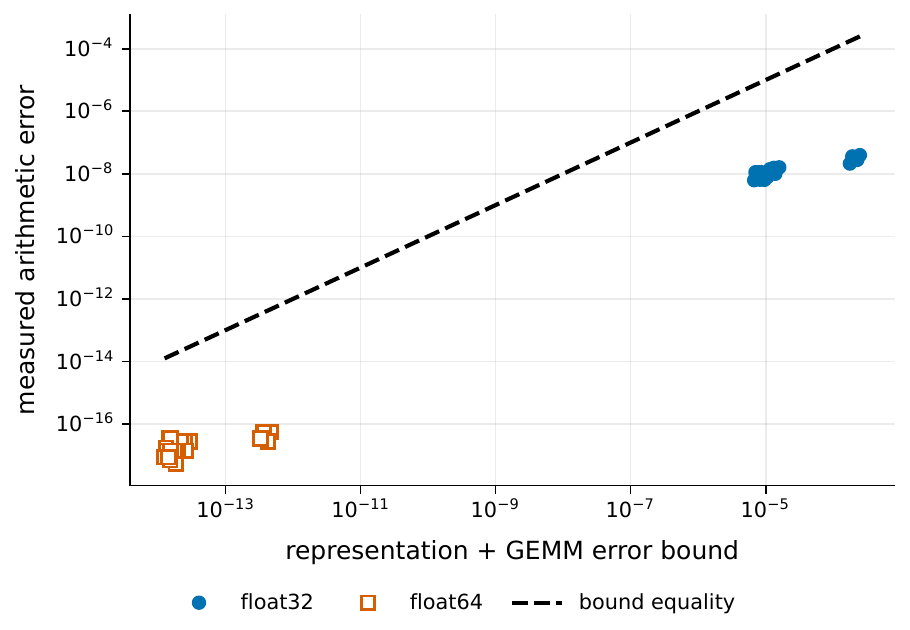}
\caption{Measured arithmetic error versus the bound including rounding of the represented factors and the GEMMs, for FP32 and FP64. The dashed line is equality.}
\label{fig:supp-arithmetic}
\end{figure}

All 48 measured arithmetic errors satisfy their bounds. The largest error-to-bound ratios are below $1.60\times10^{-3}$ for FP32 and $2.37\times10^{-3}$ for FP64. Figure~\ref{fig:supp-arithmetic} shows the gap between measured errors and their bounds. For the hierarchical representation with only diagonal dense blocks and block-relative budgets, quantization accounts for at least 86.0\% of the dominant block's error bound.

\section{Hierarchical comparisons}\label{sec:supp-hierarchical}

The following results concern representations that store only diagonal blocks densely and supplement Table~\ref{tab:exp-hierarchical} of the main paper.  All 1,320 configurations are feasible and have nonzero measured error; no dense off-diagonal fallback occurs.  The summary records satisfy the storage decomposition into diagonal blocks, admissible factors, and non-diagonal inadmissible factors.

\begin{table}[tbp]
\centering
\caption{Compression ratios $C_\varepsilon$ relative to dense FP64, including scales and maps, when only diagonal blocks are stored densely and block-relative budgets are used. The largest tested size is used in each geometry; point clouds are uniform. Ranks and bit allocations are recomputed at each tolerance.}
\label{tab:supp-h-tolerances}
\small
\begin{tabular}{lrrrrr}
\toprule
Matrix & $N$ & $C_{10^{-6}}$ & $C_{0.002}$ & $C_{0.01}$ & $C_{0.1}$ \\
\midrule
2D BEM, circle & 2048 & 21.987 & 34.127 & 39.141 & 44.262 \\
2D BEM, ellipse & 2048 & 20.416 & 32.509 & 37.107 & 42.976 \\
3D BEM, cube & 1536 & 3.134 & 9.236 & 12.524 & 19.849 \\
3D BEM, cuboid & 1536 & 3.454 & 9.626 & 12.908 & 19.596 \\
2D uniform, $1/r$ & 2048 & 5.427 & 14.583 & 19.133 & 27.668 \\
2D uniform, $\log r$ & 2048 & 6.887 & 15.636 & 19.873 & 28.451 \\
3D uniform, $1/r$ & 2048 & 2.146 & 6.806 & 9.828 & 17.370 \\
3D uniform, $\log r$ & 2048 & 1.928 & 6.219 & 8.890 & 16.005 \\
\bottomrule
\end{tabular}
\end{table}

Table~\ref{tab:supp-h-tolerances} includes the less compressible three-dimensional examples alongside planar BEM. At fixed $\varepsilon=0.01$, circle BEM's total compression increases from 6.229 at $N=256$ to 39.141 at $N=2048$.

\begin{table}[tbp]
\centering
\caption{Rank fractions $r_k/\min(m_k,n_k)$ for non-diagonal inadmissible blocks under block-relative budgets.  Medians are unweighted over blocks; maxima expose the least compressible blocks.  Matrix sizes are those in Table~\ref{tab:supp-h-tolerances}.}
\label{tab:supp-nearfield-ranks}
\small
\begin{tabular}{lrrrr}
\toprule
& \multicolumn{2}{c}{$\varepsilon=10^{-6}$} & \multicolumn{2}{c}{$\varepsilon=0.01$} \\
Matrix & Median & Maximum & Median & Maximum \\
\midrule
2D BEM, circle & 0.250 & 0.250 & 0.125 & 0.125 \\
2D BEM, ellipse & 0.203 & 0.250 & 0.094 & 0.125 \\
3D BEM, cube & 0.708 & 1.000 & 0.250 & 0.667 \\
3D BEM, cuboid & 0.583 & 1.000 & 0.250 & 0.625 \\
2D uniform, $1/r$ & 0.500 & 0.812 & 0.188 & 0.406 \\
2D uniform, $\log r$ & 0.344 & 0.656 & 0.156 & 0.344 \\
3D uniform, $1/r$ & 0.781 & 1.000 & 0.281 & 0.594 \\
3D uniform, $\log r$ & 0.844 & 1.000 & 0.281 & 0.594 \\
\bottomrule
\end{tabular}
\end{table}

The rank diagnostics in Table~\ref{tab:supp-nearfield-ranks} qualify the motivation for factorizing near-field blocks.  At $0.01$ they are often of moderate rank, but the two surface BEM cases and both three-dimensional point kernels contain full-rank inadmissible blocks at $10^{-6}$.  These remain quantized factor representations; no small-rank assumption or per-block storage advantage is imposed.

Figure~\ref{fig:supp-hmatrix-storage} contrasts the uniform three-dimensional $\log r$ kernel with circle BEM.  Relaxing $\varepsilon$ from $10^{-6}$ to $0.1$ increases dense-FP64 compression from 1.928 to 16.005 for the point kernel, compared with 21.987 to 44.262 for the circle.  The maximum rank fraction $r_k/\min(m_k,n_k)$ among non-diagonal inadmissible blocks decreases from 1 to 0.406 for the point kernel, and from 0.25 to 0.094 for the circle.  Both rank selection and bitwidth allocation contribute to these storage changes.

\begin{figure}[tbp]
\centering
\includegraphics[width=\linewidth]{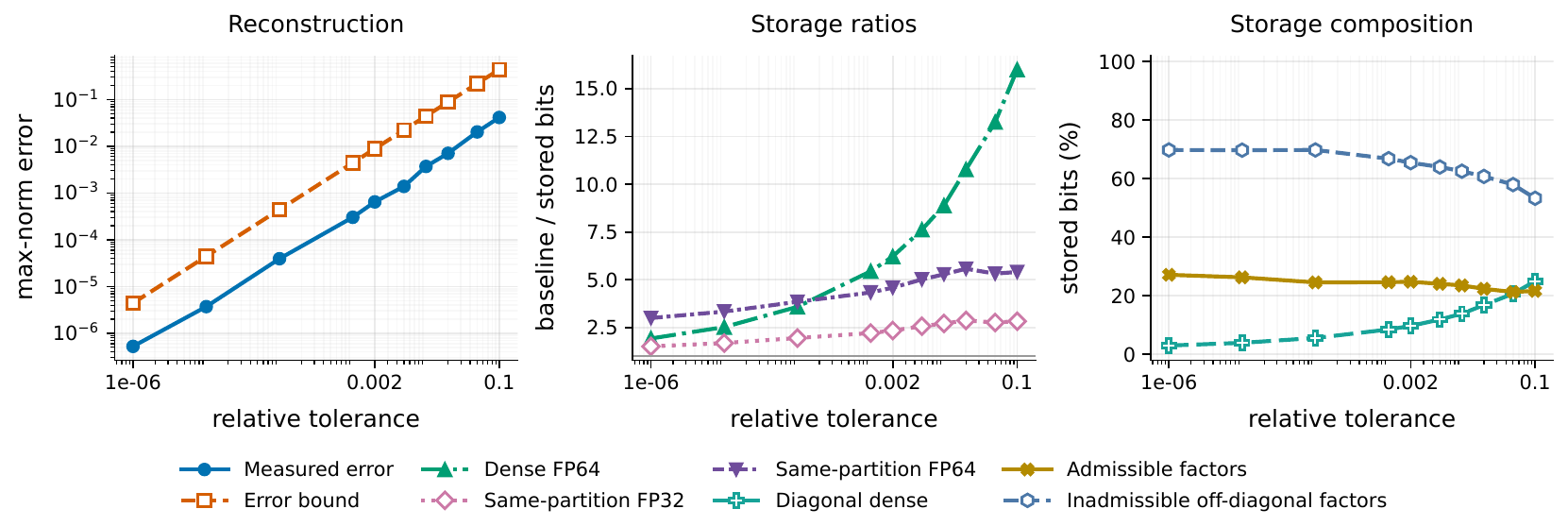}
\caption{Uniform three-dimensional $\log r$ kernel, $N=2048$, with block-relative budgets and diagonal-only dense storage.  Left: max-norm reconstruction error and its bound.  Middle: storage ratios against dense FP64 and same-partition floating factors.  Right: storage shares of diagonal blocks, admissible factors, and non-diagonal inadmissible factors.  Ranks are recomputed at each tolerance.}
\label{fig:supp-hmatrix-storage}
\end{figure}

\begin{figure}[tbp]
\centering
\includegraphics[width=0.89\linewidth]{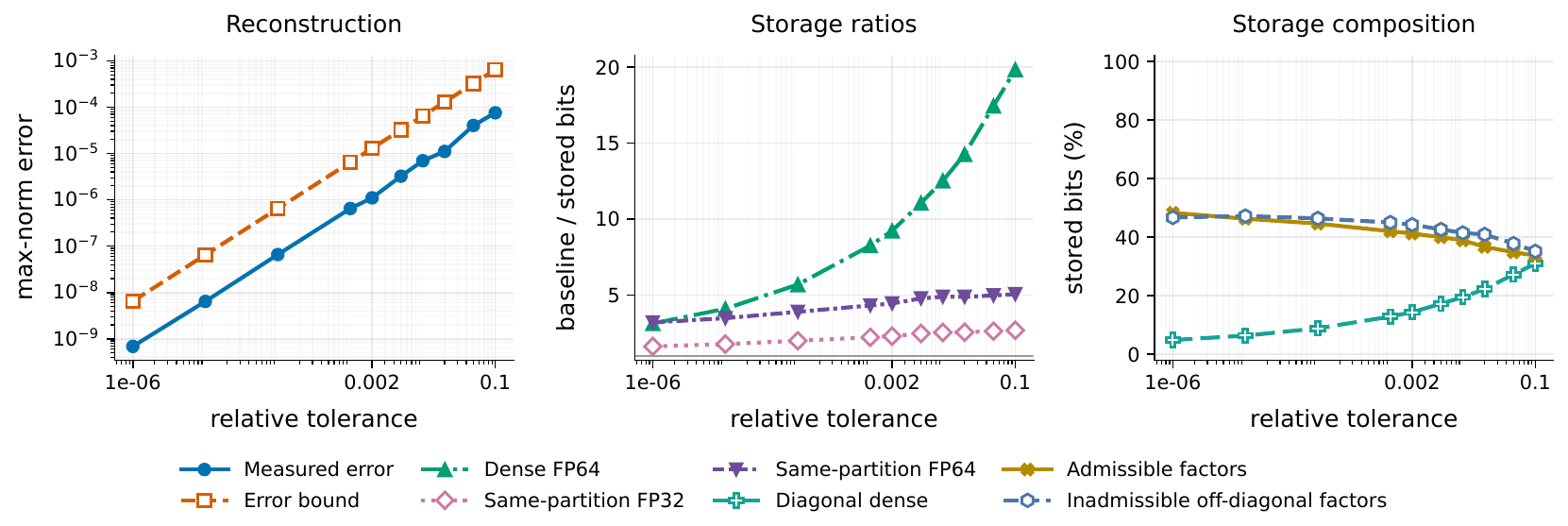}
\caption{Cube BEM, $N=1536$, with block-relative budgets and diagonal-only dense storage.  Panels use the same definitions as Figure~\ref{fig:exp-global} of the main paper.  Near-field factor storage and rank adaptation remain significant even when the complete representation is compressed.}
\label{fig:supp-nearfield}
\end{figure}

At $\varepsilon=0.01$, cube BEM assigns 19.57\% of stored bits to diagonal blocks and 41.49\% to non-diagonal inadmissible factors.  Its scale/map metadata is 17.40\% of total storage, compared with 3.69\% for circle BEM. The metadata fraction is counted within the three block-class shares; it is reported separately to show the contribution of scales and bitwidth labels, which is relatively large for short factors.

For a fixed matrix, the choice of $B_k^{\mathrm{tot}}$ changes the error allowed in each block. Table~\ref{tab:supp-policies} uses clustered two-dimensional $1/r$ at $N=2048$, $\varepsilon=0.01$.  Total budgets are $\varepsilon\|A_k\|_{\max}$ (block-relative), $\varepsilon\max(\|A_k\|_{\max},10^{-3}\|A\|_{\max})$ (hybrid), or $\varepsilon\|A\|_{\max}$ (global-relative).

\begin{table}[tbp]
\centering
\caption{Comparison of block budgets for one matrix whose representation stores only diagonal blocks densely. $e,E$ are absolute global max-norm error and error bound. All three configurations are feasible.}
\label{tab:supp-policies}
\small
\begin{tabular}{lrrr}
\toprule
Block budget & $C_{0.01}$ & $e$ & $E$ \\
\midrule
Block-relative & 20.599 & 0.9088 & 14.3087 \\
Hybrid & 20.921 & 0.9088 & 14.3087 \\
Global-relative & 44.620 & 13.2917 & 52.9982 \\
\bottomrule
\end{tabular}
\end{table}

Hybrid budgets reduce storage while leaving the measured global error and its bound unchanged. Global-relative budgets give greater compression at substantially larger measured error.

The study that retains dense inadmissible blocks uses a fixed relative Frobenius SVD target of $10^{-6}$, rank cap 64, four integer bitwidths, FP32 reconstruction, and FP64 dense fallback. In its eight largest cases at $0.01$, diagonal, near-field, and fallback dense blocks together occupy 88.2--95.6\% of stored bits. The diagonal-only dense study uses different rank selection, bitwidths, and arithmetic. The same-partition controls in the main paper isolate the storage effect of quantization.

\subsection{Fixed-bit quantization controls}\label{sec:supp-granularity}

We compare the adaptive allocation with per-tensor and per-channel fixed-bit quantization on circle BEM, uniform two-dimensional $1/r$, and uniform three-dimensional $\log r$ matrices at $N=512$ and $2048$. All three methods use the same block partition, retained SVD ranks, block-relative budgets, and available bitwidths $\{4,8,12,16,24,32\}$ as the study with only diagonal dense blocks.  For each non-diagonal block, the fixed-bit controls choose the lowest feasible width separately.  With balanced factors $P=U\Sigma^{1/2}$ and $Q=V\Sigma^{1/2}$, the respective bounds on factor-quantization error are
\begin{equation*}
 r\|P\|_{\max}\|Q\|_{\max}c(b),\qquad
 c(b)\sum_{j=1}^r\sigma_j,\qquad
 \sum_{j=1}^r\sigma_jc(b_j).
\end{equation*}
The first uses one scale for each whole factor; the latter two use one scale per retained factor column. All methods reserve the same 2\% of each block budget and use the same measured truncation error. The storage includes integer payload, FP64 scales, bitwidth labels, and the common FP64 diagonal blocks. Reconstruction is evaluated in FP64.

All 180 combinations of methods, matrices, and tolerances are feasible and satisfy their bounds on truncation and quantization error in the numerical checks. Across the 60 matched matrix and tolerance pairs, the ratio of per-channel fixed-bit to adaptive total storage has median 1.252 and range 1.052--1.682; the adaptive representation is smaller in every pair. The ratio of per-tensor fixed-bit to adaptive storage has median 1.122, and the adaptive representation is smaller in 45 of 60 pairs. Per-tensor quantization needs fewer scales and bitwidth identifiers, and its error bound uses the maxima of the actual balanced factors. Figure~\ref{fig:exp-granularity} of the main paper shows the $N=2048$ matrices; Figure~\ref{fig:supp-granularity} shows the corresponding $N=512$ cases. For each block, all three methods use the same quantization budget, but feasibility is assessed using the error bound for each method. Their measured reconstruction errors may differ.

\begin{figure}[tbp]
\centering
\includegraphics[width=0.82\linewidth]{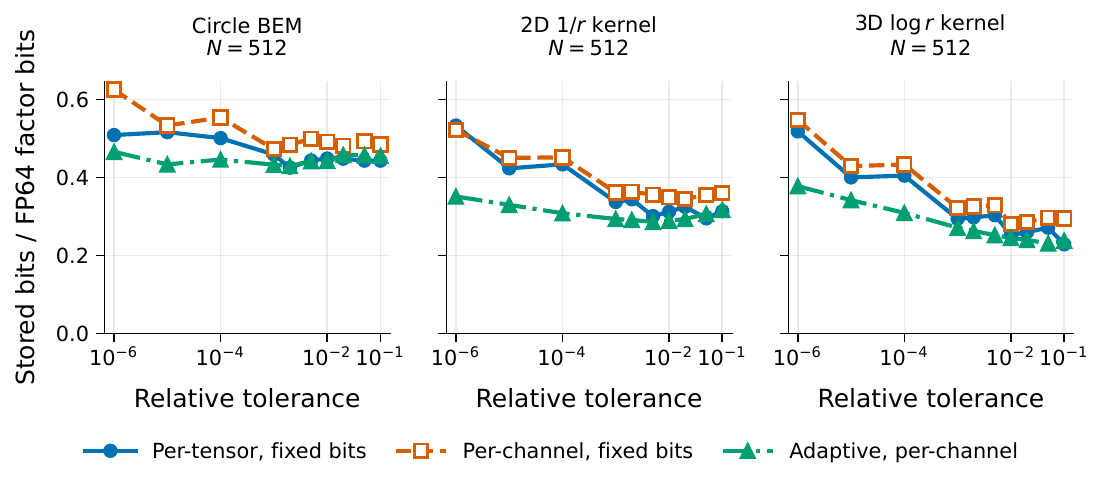}
\caption{Stored bits relative to FP64 factors on the same partition for three $N=512$ matrices with block-relative budgets, complementing Figure~\ref{fig:exp-granularity} of the main paper. Each fixed-bit control uses the lowest feasible width for each non-diagonal block; the adaptive method varies width by retained component. The common FP64 diagonal blocks, factor scales, and bitwidth labels are included. Ranks are identical across the three methods at each tolerance.} \label{fig:supp-granularity}
\end{figure}

\section{Native reconstruction validation}\label{sec:supp-native}

Synthetic outputs have orders 128, 512, and 4096, with ranks 32, 64, 128, and 256 whenever permitted by the dimensions. The 11 synthetic and 68 admissible-block inputs are each tested in six allocation modes: uniform 4-, 8-, 12-, and 16-bit assignments, prescribed mixed-bit assignments, and assignments from the proposed greedy algorithm. The adaptive budget in this suite is $0.002\sum_j\sigma_j$.

Packed reconstruction decodes factors to FP32, performs one cuBLAS GEMM per nonempty bitwidth group, and accumulates the results sequentially. The predecoded path starts from decoded FP32 factors. Strict FP32 disables TF32. The reference is the FP64 product of the saved FP32 factors.

All 1,501 reconstruction records to which the strict arithmetic error bound applies pass their stated allowances. The reconstruction allowance includes rounding in the FP32 factors, the GEMMs for each bitwidth group, and their accumulation, as specified in Section~\ref{sec:supp-arithmetic}. TF32, FP16, and shared-scale INT8 results are reported separately under their respective arithmetic and scaling models.

\section{Numerical validation of hierarchical matrix application}
\label{sec:supp-application}

The first suite evaluates $Y=AX$ with $p$ right-hand sides using the stored hierarchical partition. Non-diagonal inadmissible blocks remain dense, and low-rank blocks are applied through their factors. The ten matrices are circle and ellipse BEM at $N\in\{1024,2048\}$, cube and cuboid BEM at $N=1536$, and the four uniform point-kernel cases ($1/r$ and $\log r$ in two and three dimensions) at $N=2048$. The tolerances are $0.005,0.01,0.05$ and $p\in\{1,8,32,128\}$, giving 120 configurations. The partition, SVD rank threshold, and allocation/fallback decisions come from the stored hierarchical representations. All H-matrix variants use the same partition and ranks. Dense and fallback working blocks use FP32, and three of the 30 matrix and tolerance combinations include dense fallback for admissible blocks. The available bitwidths are $\{4,8,12,16\}$. Application arithmetic uses FP32.

The six paths comprise dense FP32 SGEMM, batched H-FP32, adaptive unpack-plus-batched, adaptive predecoded, and two matched tiled FP32 implementations using floating or packed factors. For a block $k$ with $m$ rows and column cluster $J$ of size $n$, let $X_J\in\mathbb R^{n\times p}$ contain the input rows indexed by $J$. If the block is factorized with retained rank $r$, let $F_k\in\mathbb R^{m\times r}$ and $G_k\in\mathbb R^{n\times r}$ denote its FP32 working factors (floating or decoded). The factorized block applies $F_k(G_k^T X_J)$ directly. All H-matrix paths include permutation, dense-block products, and deterministic output reduction. Packed tiled multiplication decodes weights locally for FP32 products. All six paths disable TF32 and run on the A100 described in Section~\ref{sec:supp-setup}. Each of the 120 configurations is tested with six implementations in five trials, giving 3,600 runs.

Numerical validation uses FP64 applications of the actual FP32 dense matrix, ordinary FP32 factors, or decoded FP32 quantized factors, as appropriate.  These references separate arithmetic error from error in representing the source matrix.  Here $u=2^{-24}$ is FP32 unit roundoff. For a low-rank block, the arithmetic envelope includes
\begin{equation*}
 [(1+\gamma_n(u))(1+\gamma_r(u))-1]
 |F_k|\bigl(|G_k|^T|X_J|\bigr),
\end{equation*}
with the appropriate dense-block term, output-summation rounding, and an FP64 reference allowance. The bounds assume round-to-nearest arithmetic without overflow or underflow. Let $A_k$ denote the source block and $\widehat A_k$ the exact product of its saved FP32 factors, or its FP32 dense working block. The representation error is checked using the measured block residuals $\delta_k=\|A_k-\widehat A_k\|_{\max}$ and $|(A_k-\widehat A_k)X_J| \le \delta_k\boldsymbol{1}_m (\boldsymbol{1}_n^T|X_J|)$, where $\boldsymbol{1}_m$ and $\boldsymbol{1}_n$ are vectors of ones. In all 3,600 runs, the measured arithmetic errors and the application errors bounded using measured block residuals remained below their respective bounds. The largest error-to-bound ratios were below $0.0347$ and $0.1382$, respectively.

The second suite uses the same matrices, tolerances, right-hand-side counts, and trials, with the partition and factor assignments from the diagonal-only dense representation. In this suite, diagonal blocks are stored in FP64. All non-diagonal blocks, including inadmissible blocks, are stored as quantized low-rank factors with bitwidths chosen from $\{4,8,12,16,24,32\}$. Diagonal products accumulate in FP64 before conversion to FP32 for output reduction, and factor products use FP32 arithmetic. The measured errors in all 3,600 runs remained below the corresponding arithmetic and measured-residual propagation bounds. The largest error-to-bound ratios were below $0.0292$ and $0.1895$, respectively.

\end{document}